\documentclass{article}

\usepackage{amsmath, amsthm, amscd, amssymb, xypic}

\newcommand{\R}{\mathbf R}
\renewcommand{\H}{\mathcal H}
\newcommand{\Q}{\mathbf Q}
\newcommand{\C}{\mathbf C}
\newcommand{\Z}{\mathbf Z}
\newcommand{\B}{\mathcal B}
\newcommand{\N}{\mathfrak N}
\newcommand{\F}{\mathcal F}

\newcommand{\p}{\mathfrak p}
\renewcommand{\phi}{\varphi}

\renewcommand{\k}{\textbf{\textit k}}
\newcommand{\m}{\textbf{\textit m}}
\renewcommand{\a}{\mathfrak a}

\newcommand{\eps}{\epsilon}
\renewcommand{\O}{\mathcal O}

\newcommand{\A}{\mathbf A}
\newcommand{\bs}{\backslash}

\renewcommand{\d}{\mathfrak d}
\renewcommand{\1}{\mathbf 1}

\newcommand{\bmx}{\left( \begin{matrix}}
\newcommand{\emx}{\end{matrix} \right)}

\newcommand{\I}{\tilde I}
\renewcommand{\c}{\mathfrak c}
\newcommand{\calA}{\mathcal A}

\DeclareMathOperator{\GL}{GL}  
\DeclareMathOperator{\Gal}{Gal} \DeclareMathOperator{\tr}{tr} 
  
  \DeclareMathOperator{\vol}{vol}

 \DeclareMathOperator{\PGL}{PGL} 
  
\DeclareMathOperator{\iso}{\stackrel{\sim}{\longrightarrow}} \DeclareMathOperator{\Sym}{Sym}
 \DeclareMathOperator{\Tr}{Tr} 
 \DeclareMathOperator{\Pic}{Pic} \DeclareMathOperator{\Ram}{Ram} \DeclareMathOperator{\JL}{JL}

\newtheorem{theorem}{Theorem}[section]
\newtheorem{lemma}[theorem]{Lemma}
\newtheorem{corollary}[theorem]{Corollary}

\newtheorem{proposition}[theorem]{Proposition}

\newtheorem{fact}[theorem]{Fact}

\title{Averages of central $L$-values of Hilbert modular forms with an application to subconvexity}

\author{Brooke Feigon\footnote{Email address: \texttt{bfeigon@math.toronto.edu}}\\Institute for Advanced
Study\\Einstein Drive\\Princeton, NJ 08450 \and David Whitehouse\footnote{Email address:
\texttt{dw@math.mit.edu}}\\Institute for Advanced Study\\Einstein Drive\\Princeton, NJ 08450}

\begin{document}

\maketitle

\begin{abstract}
We use the relative trace formula to obtain exact formulas for central values of certain twisted quadratic base change $L$-functions averaged over Hilbert modular forms of a fixed weight and level. We apply these formulas to the subconvexity problem for these $L$-functions. We also establish an equidistribution result for the Hecke eigenvalues weighted by these $L$-values.
\end{abstract}

\tableofcontents

\section{Introduction}

\subsection{Statement of results}

In this paper we use the relative trace formula, together with period formulas originating in work of Waldspurger \cite{waldspurger:1985}, to study central values of $L$-functions associated to Hilbert modular forms. Let $F$ be a totally real number field and let $\Sigma_\infty$ denote the set of archimedean places of $F$. Given an ideal $\N$ of $\O_F$ and a tuple of positive integers $\k = (k_v : v\in\Sigma_\infty)$ we let $\F(\N, 2\k)$ denote the set of cuspidal automorphic representations of $\PGL(2, \A_F)$ which are of exact level $\N$ and holomorphic of weight $2\k$. We recall that each $\pi \in \F(\N, 2\k)$ may be identified with a normalized holomorphic Hilbert modular newform of level $\N$, weight $2\k$ and trivial nebentypus which is an eigenfunction for all the Hecke operators.

Let $E$ be a quadratic extension of $F$ and for each $\pi \in \F(\N, 2\k)$ let $\pi_E$ denote the base change of $\pi$ to an automorphic representation of $\PGL(2, \A_E)$. Given a unitary character $\Omega$ of the idele class group $E^\times \bs \A_E^\times$ of $E$ one may consider the completed $L$-function $L(s, \pi_E\otimes\Omega)$ which satisfies a functional equation relating $s$ to $1 - s$. We note that if $\sigma_\Omega$ denotes the induction of $\Omega$ to an automorphic representation of $\GL(2, \A_F)$ then,
\[
L(s, \pi_E\otimes\Omega) = L(s, \pi \times \sigma_\Omega).
\]
The object of study in this paper is $L(1/2, \pi_E\otimes\Omega)$, the central value of this $L$-function. In particular we prove an explicit formula for $L(1/2, \pi_E\otimes\Omega)$ as one averages over $\pi$ of a fixed weight and level. As an application of this formula we establish subconvexity as $\pi$ and $\Omega$ vary in a certain range and prove an equidistribution result for the Hecke eigenvalues of such $\pi$.

Throughout this paper we make the following assumptions on the data introduced above.
\begin{itemize}
  \item $E$ is an imaginary quadratic extension of $F$,
  \item $\N$ is squarefree, each prime $\p$ dividing $\N$ is inert and unramified in $E$ and the number of primes dividing $\N$ has the same parity as the degree of the extension $[F:\Q]$, and
  \item the character $\Omega$ is trivial when restricted to $\A_F^\times$, unramified at the places of $E$ above $\N$, and at each archimedean place $\Omega_v$ has weight $m_v < k_v$.
\end{itemize}
See Section \ref{summary} below for a discussion on the relevance and seriousness of these assumptions.

The results of this paper are all derived from an exact formula for
\[
\sum_{\pi\in\F(\N, 2\k)} \frac{L(1/2, \pi_E\otimes\Omega)}{L(1, \pi, Ad)} \hat{f}_\p(\pi_\p),
\]
obtained via the relative trace formula. Here $L(s, \pi, Ad)$ is the adjoint $L$-function of $\pi$ and $\hat{f}_\p$ denotes a Hecke operator at $\p \nmid \N$.

When $\N$ is large with respect to $E$, $c(\Omega)$ (the conductor of $\Omega$) and $f_\p$ our formula simplifies considerably. The simplest version of it is given by,

\begin{theorem}
Assume that not all $k_v = 1$ and $\N$ has absolute norm larger than $d_{E/F} c(\Omega)^{h_F}$. Then,
\begin{align*}
\frac{2^{[F:\Q]}}{|\N|} \binom{2\k - 2}{\k + \m - 1} \sum_{\pi\in\F(\N, 2\k)}
\frac{L(1/2, \pi_E\otimes\Omega)}{L(1, \pi, Ad)} = {4 {|\Delta_F|}^{\frac{3}{2}}L^{S(\Omega)}(1, \eta)}
\end{align*}
where $S(\Omega)$ denotes the set of places of $F$ above which $\Omega$ is ramified and $\eta$ is the quadratic idele class character of $F$ associated to $E$.
\label{thmintro1}
\end{theorem}

For the full formula see Section \ref{mainresults} and for any undefined notation see Section \ref{notation}.

Over $\Q$ we can express this formula more classically. We identify $\F(N, 2k)$ with the set of normalized modular newforms of level $N$ and weight $2k$ which are Hecke eigenforms. Let $E = \Q(\sqrt{-d})$ be an imaginary quadratic extension of $\Q$ of discriminant $-d$. We take $\Omega$ as before and let $g_\Omega$ denote the modular form of weight $2|m| + 1$, level $d c(\Omega)$ and nebentypus $\chi_{-d}$ associated to $\Omega$. For $f \in \F(N, 2k)$ we take the Rankin-Selberg $L$-function, $L(s, f \times g_\Omega)$ which has functional equation relating $s$ to $2k + 2|m| + 1 - s$. Using the relationship between $L(1, \pi_f, Ad)$, where $\pi_f$ denotes the automorphic representation generated by $f$, and the square of the Petersson norm of $f$,
\[
(f, f) = \int_{\Gamma_0(N) \bs \H} |f(x + iy)| y^{2k} \ \frac{dx \ dy}{y^2}
\]
we can rewrite Theorem \ref{thmintro1} in the following way: Assume $\Omega$ is not quadratic, $k>1$ and $N > d c(\Omega)$, then
\[
\frac{(2k-2)!u_{-d} \sqrt{d}L_{S(\Omega)}(1, \chi_{-d})}{2\pi (4\pi)^{2k-1}}  \sum_{f\in\F(N, 2k)}
\frac{L_{fin}(k + |m| + \frac 12, f \times g_\Omega)}{(f, f)} = h_{-d}
\]
where $\chi_{-d}$ is the quadratic Dirichlet character of discriminant $-d$, $h_{-d}$ is the class number of $E$ and $u_{-d} = \# \O_E^\times/\{\pm 1\}$. See Section \ref{classical} for further details.

We note that over $\Q$ and with $\Omega$ trivial an asymptotic version of Theorem \ref{thmintro1} has been known for a while. First by Duke \cite{duke:1995} for prime level and weight 2 and then by Iwaniec, Luo and Sarnak \cite{iwaniec:2000}, \cite{iwaniec:2000a} for squarefree level and higher weight. An exact formula has been established by Michel and Ramakrishnan in the case $F=\Q$ and $\Omega$ is a character of the ideal class group of $E$. Their work uses Gross' formula together with a geometric argument in the weight two case, and the theta correspondence in higher weight; see \cite{ramakrishnan:exact}.

By an extension of Theorem \ref{thmintro1} to include Hecke operators, we obtain a result that includes a restriction at a prime $\p \nmid \N$. Let $\pi = \otimes_v \pi_v \in \F(\N, 2\k)$. We let $\{ \alpha_\p, \alpha^{-1}_\p\}$ denote the Satake parameters of $\pi_\p$ and set $a_\p(\pi) = \alpha_\p + \alpha^{-1}_\p$. We recall that $a_\p(\pi) \in [-2, +2]$ by Ramanujan's conjecture; \cite{blasius:2006}. The distribution of the $a_\p(\pi)$ has been considered by Sarnak \cite{sarnak:1987} and Serre \cite{serre:1997}. The spherical Plancherel measure on $\PGL(2, F_\p)$ is given by
$$
\mu_\p = \frac{q_\p + 1}{(q_\p^{\frac 12} + q_\p^{-\frac 12})^2 - x^2} \frac{\sqrt{4 - x^2}}{2\pi} dx
$$
on $[-2, 2]$, here $q_\p$ denotes the order of the residue field at $\p$. Serre has proven that when $F = \Q$,
$$
\{ a_\p(\pi) : \pi \in \F(\N, 2\k) \}
$$
become equidistributed with respect to $\mu_\p$ as $\N \to \infty$. We prove a variant of this result where we include a weighting by $L^\p(1/2, \pi_E\otimes\Omega)$.

\begin{theorem}
For any $J \subset [-2, +2]$ we have
\begin{align*}
&\lim_{\N \to \infty} \frac{1}{|\N|} \sum_{\substack{\pi\in\F(\N, 2\k)\\a_\p(\pi)\in J}}
\frac{L^\p(1/2, \pi_E\otimes\Omega)}{L^\p(1, \pi, Ad)}
\\
=& 4 |\Delta_F|^{\frac 32} \frac{1}{2^{[F:\Q]}} \binom{2\k - 2}{\k + \m -1}^{-1}
L^{S(\Omega)\cup\{\p\}}(1, \eta) L(2, 1_{F_\p}) \mu_\p(J).
\end{align*}
\end{theorem}

We note that in the case that $F = \Q$ and $\Omega$ is trivial, we recover the main result of \cite{ramakrishnan:2005}. A similar result has been obtained by Royer \cite{royer:2000} for the single $L$-value $L(1/2, \pi)$ averaged over modular forms $\pi$ of level $N$ and weight 2.

We remark that one could consider the average by normalizing by $|\F(\N, 2\k)|$ rather than $|\N|$. In order to have a finite limit with this normalization we need to add the technical restriction thatb $|\N|\prod_{\p | \N}(1-\frac{1}{|\p|})\sim |\N|$.  For $F=\Q$ this condition reduces to  $\phi(N) \sim N$ where $\phi$ is the Euler totient function.  Using the well known fact that $|\F(N, 2k)| \sim  \frac{2k-1}{12}\phi(N)$ as $N\rightarrow \infty$ we get the following statement.

\begin{corollary}
Let $F=\Q$ and $J \subset [-2, +2]$. Then
\[
\lim_{N \to \infty} \frac{1}{|\F(N, 2k)|} \sum_{\substack{\pi\in \F(N, 2k) \\ a_p(\pi)\in J}}
\frac{L^p(1/2, \pi_E\otimes\Omega)}{L^p(1, \pi, Ad)}
\]
is equal to
\[
\frac{24}{2k-1}  \binom{2k - 2}{k + m - 1}^{-1}
L^{S(\Omega)\cup\{p\}}(1, \eta) L(2, 1_{\Q_p}) \mu_{p}(J)
\]
where the limit is taken over squarefree $N$ such that $\phi(N) \sim N$ and each prime dividing $N$ is inert and unramified in $E$ and does not divide $c(\Omega)$.
\end{corollary}

Finally we apply our work to the problem of subconvexity. Using a version of Theorem \ref{thmintro1} that is also valid for smaller $\N$, combined with the non-negativity of $L(1/2, \pi\times \sigma_\Omega)$, established in \cite{jacquet:2001}, and an upper bound for $L(1, \pi, Ad)$, we get the following theorem.

\begin{theorem}
Fix a totally real number field $F$ and a CM extension $E$ of $F$. Let $\N$ be a squarefree ideal in $\O_F$ such that the number of primes dividing $\N$ has the same parity as $[F: \Q]$ and such that each prime of $F$ dividing $\N$ is inert and unramified in $E$. Let $\Omega$ be a character of $\A_F^\times E^\times \bs \A_E^\times$ which is unramified above $\N$ and has weights at the archimedean places strictly less than $\k$. Then for any $\epsilon>0$,
\[
L_{fin}(1/2, \pi \times \sigma_\Omega) \ll_{F, E, \k, \eps} |\N|^{1+\eps}c(\Omega)^\eps + |\N|^{\eps}c(\Omega)^{\frac{1}{2}+\epsilon},
\]
for all $\pi \in \F(\N, 2\k)$.

Hence for $0 \leq t < \frac 16$ and $\eps > 0$,
\[
L_{fin}(1/2, \pi \times \sigma_\Omega)\ll_{F, E, \k, \eps} (c(\Omega)|\N|)^{\frac{1}{2}-t},
\]
for $\pi\in\F(\N, 2\k)$ with $\N$ such that
\[
c(\Omega)^{\frac{2 t + \eps}{1-(2 t + \eps)}}\leq |\N|\leq c(\Omega)^{\frac{1-(2t+\eps)}{1+2 t + \eps}}.
\]
\end{theorem}

This result clearly beats the convexity bound $L_{fin}(1/2, \pi\times \sigma_\Omega) \ll_\eps (d_{E/\Q}c(\Omega)N)^{\frac{1}{2} + \eps}$ for all $\eps>0$. Similar results have been obtained in \cite{ramakrishnan:exact}, where $\Omega$ is a character of the ideal class group of $E$ and $N$ and $E$ vary. We remark that Michel and Harcos (\cite{harcos:2006} and \cite{michel:2004}) have proven subconvexity in the level aspect for $L_{fin}(1/2, \pi_1\times\pi_2)$ where $\pi_1$ and $\pi_2$ are cusp forms on $\GL(2)/\Q$ with $\pi_2$ fixed. We also mention the work of Cogdell, Piatetski-Shapiro and Sarnak \cite{cogdell:2003} which proves subconvexity for the central value of a fixed Hilbert modular form twisted by a ray class character.

\subsection{About the proof}

Before continuing we give some background about the tools used in the proofs of the results contained in this paper.

\subsubsection{Waldspurger's result}

An important result of Waldspurger \cite{waldspurger:1985} relates $L(1/2, \pi_E\otimes\Omega)$ to period integrals of automorphic forms over the torus $E^\times$. More precisely let $X(\pi, E)$ be the set of isomorphism classes of quaternion algebras $D/F$ such that $E \hookrightarrow D$ and such that $\pi$ comes from an automorphic representation $\pi^D$ of $D^\times$ via the Jacquet-Langlands correspondence. We note that $X(\pi, E)$ is finite, since $D$ must be isomorphic to the matrix algebra at all places where $\pi$ is unramified, and non-empty, since $M(2, F)\in X(\pi, E)$.

Take $D \in X(\pi, E)$ and let $\phi$ be an element in the space of $\pi^D$. Thus, $\phi$ is a function on $D^\times \bs D^\times(\A_F)$ and we can define period integrals by
$$
P_D(\phi) = \int_{E^\times \A_F^\times \bs \A_E^\times} \phi(t) \Omega^{-1}(t) \ dt,
$$
with $\A_E^\times$ viewed as a subgroup of $D^\times(\A_F)$ via the inclusion $E \hookrightarrow D$. In \cite[Proposition 7]{waldspurger:1985}, $|P_D(\phi)|^2$ is expressed in terms of $L(1/2, \pi_E\otimes\Omega)$. For suitable choices of measure Waldspurger proves that,
$$
\frac{|P_D(\phi)|^2}{(\phi, \phi)} = \frac{L(1/2, \pi_E\otimes\Omega)}{2 L(1, \pi, Ad)} \prod_v \alpha_v(E_v, \phi_v, \Omega_v),
$$
where $\alpha_v(E_v, \phi_v, \Omega_v)$ are local period integrals which are equal to 1 for almost all $v$. Subsequent refinements of Waldspurger's result, \cite{gross:1987}, \cite{zhang:2001}, \cite{xue:2006}, \cite{popa:2006}, \cite{me:periods}, have sought to compute these local period integrals so as to provide an exact formula relating the $L$-value and the period integral. Under certain additional ramification conditions these results take the form,
\begin{equation}
\frac{|P_D(\phi_\pi)|^2}{(\phi_\pi, \phi_\pi)} = C(E, \pi, \Omega) L(1/2, \pi_E\otimes\Omega),
\label{periodl}
\end{equation}
where $D$ is a suitable element of $X(\pi, E)$ and $\phi_\pi \in \pi^D$ is a test vector defined by Gross and Prasad \cite{gross:1991}. The constant $C(E, \pi, \Omega)$ is of the form,
$$
C(E, \pi, \Omega) = \frac{1}{L(1, \pi, Ad)} \frac{\sqrt{\Delta_F}}{2 \sqrt{c(\Omega) \Delta_E}} \prod_v C_v(E, \pi, \Omega)
$$
where the product is taken over the places $v$ of $F$ which are ``bad" for $\pi$ or $\Omega$ and $C_v(E, \pi, \Omega)$ consists of certain local $L$-factors.

\subsubsection{Relative trace formula}

Having fixed the extension $E/F$ let $D/F$ be a quaternion algebra such that $E\hookrightarrow D$. We set $G$ equal to the group $PD^\times$ over $F$ and let $T$ be the torus in $G$ such that $T(F)$ is equal to the image of $E^\times$ in $G(F)$.

Let $f\in C^\infty_c(G(\A_F))$. Then we have the map
$$
R(f) : L^2(G(F)\bs G(\A_F)) \to L^2(G(F)\bs G(\A_F))
$$
given by
$$
(R(f)\phi)(x) = \int_{G(\A_F)} f(y) \phi(x y) \ dy.
$$
$R(f)$ is an integral operator with kernel
$$
K_f(x, y) = \sum_{\gamma\in G(F)} f(x^{-1}\gamma y).
$$
When $D$ is a division algebra, as a representation of $G(\A_F)$,
$$
L^2(G(F)\bs G(\A_F)) = \bigoplus_{\pi\in\mathcal A(G)} \pi
$$
with the sum taken over the set of irreducible automorphic representations $\mathcal A(G)$ of $G(\A_F)$. Since $R(f)$ preserves each of the spaces $\pi$, the kernel has a spectral expansion
$$
K_f(x, y) = \sum_{\pi\in\mathcal A(G)} \sum_{\phi \in \B(\pi)} (R(f)\phi)(x) \overline{\phi(y)},
$$
where $\B(\pi)$ denotes an orthonormal basis of the space of $\pi$.

Let $I(f)$ be the distribution defined by
$$
I(f) = \int_{T(F)\bs T(\A_F)} \int_{T(F)\bs T(\A_F)} K_f(t_1, t_2) \Omega(t_1^{-1}t_2) \ dt_1 \ dt_2.
$$
The spectral expansion for $K_f(x, y)$ gives,
$$
I(f) = \sum_{\pi\in\mathcal A(G)} \sum_{\phi \in \B(\pi)} \int_{T(F)\bs T(\A_F)} (R(f)\phi)(t_1)\Omega^{-1}(t_1) \ dt_1 \overline{\int_{T(F)\bs T(\A_F)} \phi(t_2)\Omega^{-1}(t_2) \ dt_2}.
$$
From the geometric expansion for the kernel, and after interchanging summation and integration,
\begin{equation}
I(f) = \sum_{\gamma \in T(F) \bs G(F) / T(F)} \vol(T_\gamma(F)\bs T_\gamma(\A_F)) \int_{T_\gamma(\A_F) \bs (T(\A_F)\times T(\A_F))} f(t_1^{-1} \gamma t_2) \Omega(t_1^{-1} t_2) \ dt_1 \ dt_2,
\label{geomexp}
\end{equation}
where, for $\gamma \in G(F)$,
$$
T_\gamma = \left\{ (t_1, t_2) \in T\times T : t_1^{-1}\gamma t_2 = \gamma \right\}.
$$

We now return to the setting of the previous section. Having fixed $\N$ we let $D$ be the quaternion algebra over $F$ which is ramified at the primes dividing $\N$ and all the infinite places of $F$. Since the extension $E/F$ is inert at all the places of ramification for $D$ there exists an embedding $E \hookrightarrow D$. For each $\pi \in \F(\N, 2\k)$ we let $\phi_\pi$ be an element in the space of $\pi^D$ as in \eqref{periodl} above. We can choose a test function $f = f_{\N, \k}$ in $C^\infty_c(G(\A_F))$, such that the operator $R(f)$ projects $L^2(G(F)\bs G(\A_F))$ onto the span of the set $\{ \phi_\pi \in \pi^D : \pi \in \F(\N, 2\k) \}$. (This isn't quite true if all $k_v = 1$, in this case the function $f_{\N, \k}$ may also pick out some 1-dimensional representations as well.) Hence for this function,
$$
I(f_{\N, \k}) =  \sum_{\pi\in\F(\N, 2\k)} \frac{|\int_{T(F)\bs T(\A_F)} \phi_\pi(t)\Omega^{-1}(t) \ dt|^2}{(\phi_\pi, \phi_\pi)}.
$$

Applying the identity \eqref{periodl} for each $\pi\in\F(\N, 2\k)$ we get,
$$
I(f_{\N, \k}) = C(E, \Omega, \k, \N) \sum_{\pi \in \F(\N, 2\k)} \frac{L(1/2, \pi_E\otimes\Omega)}{L(1, \pi, Ad)},
$$
with $C(E, \Omega, \k, \N)$ an explicit constant. From the geometric expansion for $I(f_{\N, \k})$, \eqref{geomexp}, we obtain a closed expression for this average in terms of orbital integrals of the function $f_{\N, k}$ over double cosets. Furthermore one can show that only finitely many $T(F)$ double cosets contribute to the sum. The results of this paper then stem from calculations of these geometric terms. When the level $\N$ is large we show that only the identity double coset contributes which gives an exact formula for the average (Theorem \ref{thmlargelevel}), when the character $\Omega$ is everywhere unramified we can compute all the necessary orbital integrals (Theorem \ref{thmunramified}) and when the character $\Omega$ ramifies we are able to bound the terms on the geometric side of the trace formula which we can use towards the subconvexity problem for these $L$-values (Theorem \ref{thmsub}).

\subsubsection{The case of modular forms of weight 2}

When $F=\Q$ the results of this paper are rephrased classically in Section \ref{classical}. Here we illustrate our methods for modular forms of weight 2 and prime level in more classical language.

We fix an imaginary quadratic extension $E = \Q(\sqrt{-d})$ and take $N$ to be a prime which is inert and unramified in $E$. We take $\Omega$ to be a character of $\Pic(E)$, the ideal class group of $E$. Let $D$ be the quaternion algebra over $\Q$ which is ramified at $N$ and $\infty$. We fix an embedding $E \hookrightarrow D$ and take $R$ to be a maximal order in $D$ such that $R\cap E = \O_E$. Let $X$ denote the finite set of left equivalence classes of right $R$-ideals; see \cite[Section 1]{gross:1987}. Given an ideal $\a$ in $E$ we may form the right $R$-ideal $\a R$ in $D$. In this way we get a well defined map $\iota : \Pic(E) \to X$. Let $F(X)$ denote the space of complex valued functions on $X$ which can be endowed with a natural inner product and an action of Hecke operators.

The Jacquet-Langlands correspondence, in this case, gives a Hecke-equivarient isomorphism $\JL : F(X) \iso M_{2}(N)$ of $F(X)$ with the space of modular forms of weight 2 and level $N$; see \cite[Section 5]{gross:1987}. We set $S(X) = \JL^{-1}(S_2(N))$ where $S_2(N)$ denotes the space of cusp forms of level $N$ and weight 2. The orthogonal complement of $S(X)$ in $F(X)$ is the space of constant functions.

We identify $\F(N, 2)$ with the set of normalized eigenforms in $S_2(N)$. For $f\in\F(N, 2)$ let $f'\in F(X)$ be such that $\JL(f') = f$. Then the period relation for the central $L$-value can be expressed as,
$$
\frac{L(1, f \times g_\Omega)}{(f, f)} = C(N, E) \frac{\left|\sum_{y\in\Pic(E)} f'(\iota(y)) \Omega(y) \right|^2}{(f', f')},
$$
where $g_\Omega$ is the $\theta$-series associated to $\Omega$ and $L(1, f \times g_\Omega)$ denotes the central value of the Rankin-Selberg $L$-function of $f$ with $g_\Omega$. Here $C(N,E)$ is an explicit constant depending only on $N$ and $E$. Thus,
$$
\sum_{f\in \F(N, 2)} \frac{L(1, f \times g_\Omega)}{(f, f)} = C(N, E) \sum_{f\in\F(N, 2)} \frac{\left|\sum_{y\in\Pic(E)} f'(\iota(y)) \Omega(y) \right|^2}{(f', f')}.
$$
Since the orthogonal complement of $S(X)$ is spanned by the constant functions we see that if $\B$ is {\em any} orthonormal basis of $F(X)$ then
$$
\sum_{f\in \F(N, 2)} \frac{L(1, f \times g_\Omega)}{(f, f)} + C(N, E) \delta_{\Omega} \frac{\# \Pic(E)^2}{(\1, \1)} = C(N, E) \sum_{h\in\B} \left|\sum_{y\in\Pic(E)} h(\iota(y)) \Omega(y) \right|^2,
$$
where $\1$ denotes the function which is identically $1$ on $X$, and
$$
\delta_{\Omega} =
\left\{
  \begin{array}{ll}
    1, & \hbox{if $\Omega$ is trivial;} \\
    0, & \hbox{otherwise.}
  \end{array}
\right.
$$
For each $x\in X$ let $\delta_x \in F(X)$ be the function defined by
$$
\delta_{x}(y) =
\left\{
  \begin{array}{ll}
    1, & \hbox{if $y = x$;} \\
    0, & \hbox{otherwise.}
  \end{array}
\right.
$$
Now suppose we take for $\B$ the set $\left\{\delta_{x}/\|\delta_x\| : x \in X\right\}$. Then,
\begin{align*}
\sum_{h\in\B} \left|\sum_{y\in\Pic(E)} h(\iota(x)) \Omega(x) \right|^2 &= \sum_{x\in X} \sum_{y_1\in\Pic(E)} \sum_{y_2\in\Pic(E)} \frac{\delta_x(\iota(y_1)) \delta_x(\iota(y_2))}{\|\delta_x\|^2} \Omega(y_1 y_2^{-1})\\
&= \sum_{y_1\in\Pic(E)} \sum_{y_2\in\Pic(E)} \left( \frac{\delta_x(\iota(y_1)) \delta_x(\iota(y_2))}{\|\delta_x\|^2} \right) \Omega(y_1 y_2^{-1}).
\end{align*}
For $y_1, y_2 \in \Pic(E)$ clearly,
$$
\sum_{x\in X} \delta_x(\iota(y_1)) \delta_x(\iota(y_2)) =
\left\{
  \begin{array}{ll}
    1, & \hbox{if $\iota(y_1) = \iota(y_2)$;} \\
    0, & \hbox{otherwise.}
  \end{array}
\right.
$$
Let $\a_y$ denote an ideal of $E$ which lies in the class of $y\in\Pic(E)$. Then $\iota(y_1) = \iota(y_2)$ if and only if there exists $\gamma$, well defined in $E^\times \bs D^\times / E^\times$, such that $\gamma \a_{y_1} R = \a_{y_2} R$. For each $\gamma\in E^\times \bs D^\times / E^\times$ let,
$$
S(\gamma) = \left\{ (y_1, y_2) \in \Pic(E) \times \Pic(E) : \gamma \a_{y_1} R = \a_{y_2} R \right\}.
$$
Hence,
$$
\sum_{h\in\B} \left|\sum_{x\in\Pic(E)} h(\iota(x)) \Omega(x) \right|^2 = \sum_{\gamma \in E^\times \bs D^\times / E^\times} \sum_{(y_1, y_2) \in S(\gamma)} \frac{\Omega(y_1 y_2^{-1})}{\|\delta_{\iota(y_1)}\|^2},
$$
and so we deduce that
$$
\sum_{f\in \F(N, 2)} \frac{L(1, f \times g_\Omega)}{(f, f)} + C(N, E) \delta_{\1, \Omega} \frac{\# \Pic(E)^2}{(\1, \1)}
$$
equals
$$
C(N, E) \left( \sum_{\gamma \in E^\times \bs D^\times / E^\times} \sum_{(y_1, y_2) \in S(\gamma)} \frac{\Omega(y_1 y_2^{-1})}{\|\delta_{\iota(y_1)}\|^2} \right).
$$
This is precisely the same identity, albeit without the adelic language, obtained via the relative trace formula. Furthermore it is clear that if $\iota : \Pic(E) \to X$ is injective then $S(\gamma)$ is empty unless $\gamma$ is the identity coset in $E^\times\bs D^\times /E^\times$. The fact that $\iota$ is injective for $N$ sufficiently large then implies the stability result (Theorem \ref{thmlargelevel}) for the sum of the $L$-values which we obtain via the relative trace formula. For the application to subconvexity (Theorem \ref{thmsub}) we need to provide bounds on the terms appearing in the geometric expansion. These we obtain from local calculations of the orbital integrals.

A similar situation has been studied in this setting by Michel and Venkatesh \cite[Section 3]{michel:2007}. In their work the form $f$ is fixed and $\Omega$ is allowed to vary over all characters of $\Pic(E)$. By averaging the period formula one obtains,
$$
\sum_{\Omega\in\widehat{\Pic(E)}} \frac{L(1, f\times g_\Omega)}{(f, f)} = C(N, E) \sum_{y\in \Pic(E)} \frac{|f'(\iota(y))|^2}{(f', f')}.
$$
Analogously to the case considered here, when the discriminant of $E$ is large relative to the level of $f$ this formula can be made exact; see \cite[Remark 3.1]{michel:2007}.

\subsection{About the paper}

\subsubsection{Summary of conditions}
\label{summary}

Throughout this paper we enforce certain conditions on the level $\N$ and weights $2 \k$ of our Hilbert modular forms as well as on the extension $E$ and the character $\Omega$. We now describe these conditions.

The ideal $\N$ of $F$ is assumed to be squarefree and such that each prime divisor of $\N$ is inert and unramified in $E$. This is, for the most part, a technical assumption so that we avoid the problem of having to deal with oldforms. This condition can most likely be removed with some extra work. We also assume that the number of prime divisors of $\N$ has the same parity as $[F:\Q]$, this is a necessary condition as it ensures that the sign of the functional equation of $L(s, \pi_E\otimes\Omega)$ is $+1$.

We place certain restrictions on the character $\Omega$ as well. We assume that $\Omega$ is unramified above $\N$. This assumption is needed in the work of Gross and Prasad \cite{gross:1991} which provides the test vector $\phi_\pi$ which appears in \eqref{periodl}. We also assume, at the archimedean places, that the weights of $\Omega$ are strictly smaller than the weights $\k$ of the Hilbert modular forms. This ensures that the quaternion algebra $D$ which appears in \eqref{periodl} is ramified at the archimedean places of $F$. Without this assumption the function $f_{\N, \k}$ would not be compactly supported and one would not obtain a {\em finite} closed formula for the average $L$-values. In this case it is likely that one could work out an asymptotic version of the formulas in this paper. This would require some additional archimedean calculations similar to \cite[Section 2]{ramakrishnan:2005}.

Finally it would also be interesting to consider the case of Maass forms. It would seem that again with some additional archimedean calculations this could be carried out.

\subsubsection{Outline of the paper}

This paper is set up as follows. We begin by introducing our notation and defining our measures. In
the following section we pick our test functions and compute their spectral expansion in the
relative trace formula. In the fourth section we explicitly compute the geometric side of the
relative trace formula for these test functions. While, a priori, there are an infinite number of
orbital integrals that appear in the geometric expansion, we show that for our test functions only
a finite number of terms are nonzero. This is what allows us to get an exact, rather than
asymptotic, formula.

In the fifth section we compute a distribution on the Hecke algebra which appears in the geometric
expansion of the relative trace formula. In the last section we combine the spectral and geometric
calculations to get our main formulas. We also establish our application to subconvexity and
rewrite our formulae over $\Q$ in classical language.

The original motivation for our work came from trying to understand \cite{ramakrishnan:2005}, which
applies the relative trace formula of \cite{jacquet:1986} to average values of base change
$L$-functions. We found that by integrating over nonsplit rather than split tori, and applying
\cite{me:periods}, we obtain more general and exact results. Furthermore, by working on quaternion
algebras we avoid having to deal with contributions from oldforms.

\subsection{Acknowledgements}

We wish to thank Peter Sarnak for his valuable suggestions and encouragement. We thank Dinakar Ramakrishnan for his useful comments, as well as furnishing us with early versions of his work with Philippe Michel \cite{ramakrishnan:exact}. We also thank Dipendra Prasad for helpful conversations and the Institute for Advanced Study for a stimulating mathematical environment. Finally, we thank the referees for their helpful comments.

The first author was supported by NSF grant DMS-0111298 and the Natural Sciences and Engineering Research Council of Canada. The second author was supported by NSF grants DMS-0111298 and DMS-0758197. Any opinions, findings and conclusions or recommendations expressed in this material
are those of the authors and do not necessarily reflect the views of the National Science
Foundation.

\section{Notation}
\label{notation}

Let $F$ be a totally real number field of discriminant $\Delta_F$ and class number $h_F$. For a
finite place $v$ of $F$, $\varpi_v$ denotes a uniformizing parameter in $F_v$ and $q_v$ denotes the
cardinality of the residue class field at $v$. The ring of integers in $F_v$ is denoted by
$\O_{F_v}$, the units by $U_{F_v}$ and for $n>0$ we set $U_{F_v}^n = 1 + \varpi_v^n \O_{F_v}$. For
an ideal $\mathfrak a$ of $\O_F$, we denote by $|\mathfrak a|$ the absolute norm of $\mathfrak a$.
We denote by $\Sigma_\infty$ the set of infinite places of $F$. We take $\N$ to be a squarefree
ideal in $\O_F$ such that the number of primes dividing $\N$ has the same parity as $[F:\Q]$. For
each place $v\in\Sigma_{\infty}$ we fix an integer $k_v\geq 1$. We denote by $\F(\N, 2\k)$ the
set of cuspidal automorphic representations of $\PGL(2, \A_F)$ of level $\N$ and weight $2\k = (2k_v)$.

We now take a CM extension $E/F$ such that each prime of $F$ dividing $\N$ is inert and unramified
in $E$. We let $\mathfrak D_{E/F}$ be the different of $E/F$, $\d_{E/F}$ be the discriminant of
$E/F$ and  $d_{E/F}=|\d_{E/F}|$. We denote by $\eta$ the quadratic character of
$F^\times\bs\A_F^\times$ associated to $E/F$ by class field theory and by $N$ the norm map from $E$ to $F$. For a
place $v$ of $F$ we denote by $E_v = E \otimes_F F_v$, by $\O_{E_v}$ the integral closure of
$\O_{F_v}$ in $E_v$ and by $U_{E_v}=\O_{E_v}^\times$. We denote the action of the non-trivial
element in $\Gal(E/F)$ on $\alpha \in E_v$ by $\bar{\alpha}$.

We take a unitary character $\Omega : E^\times\bs\A_E^\times \to \C^\times$ such that
$\Omega|_{\A_F^\times}$ is trivial. We assume that $\Omega$ is unramified above $\N$ and that at
each place $v\in\Sigma_{\infty}$, $\Omega_v$ has the form
$$
z \mapsto \left(\frac{z}{\bar{z}}\right)^{m_v}
$$
with $|m_v|$ strictly less than $k_v$. We set $\m = (m_v)$. At each finite place $v$ of $F$ we denote by $n(\Omega_v)$ the
smallest non-negative integer such that $\Omega$ is trivial on $(\O_F + \varpi_v^{n(\Omega_v)} \O_E)^\times$. We denote by $\mathfrak c(\Omega)$ the
norm of the conductor of $\Omega$ in $F$ and by $c(\Omega)$ the absolute norm of $\mathfrak
c(\Omega)$. We note that $c(\Omega)=\prod_{v< \infty}q_v^{2n(\Omega_v)}$. We use $S(\Omega)$ to denote the set of places of $F$ above which $\Omega$ is ramified.

We define
$$
\binom{2\k - 2}{\k + \m - 1} = \prod_{v\in\Sigma_\infty} \binom{2k_v - 2}{k_v + m_v - 1}.
$$

Having fixed $\N$, we let $D$ be the quaternion algebra defined over $F$ which is ramified precisely at the
infinite places of $F$ and the places dividing $\N$. We fix an embedding $E\hookrightarrow D$. We
let $Z$ denote the center of $D$ and denote by $G$ the group $Z \bs D^\times$ viewed as an algebraic
group defined over $F$. We denote by $N_D : D^\times \to F^\times$ the reduced norm. At each finite
place $v$ of $F$ we fix a maximal order $R_v$ in $D_v$ such that $R_v\cap E_v = \O_{F_v} +
\varpi_v^{n(\Omega_v)}\O_{E_v}$. We note that $R_v$ is well defined up to $E_v^\times$ conjugacy.

All $L$-functions are completed.

\subsection{Normalization of measures}

We fix an additive character $\psi$ of $F\bs\A_F$ which, for the sake of convenience, we take to be
$\psi = \psi_0 \circ \tr_{F/\Q}$ where $\psi_0$ denotes the standard character on $\Q\bs\A_\Q$. For
a place $v$ of $F$ we take the additive Haar measure $dx$ on $F_v$ which is self-dual with respect
to $\psi_v$. On $F_v^\times$ we take the measure
$$
d^\times x = L(1, 1_{F_v}) \frac{dx}{|x|_v}.
$$
We define measures on $E_v$ and $E_v^\times$ similarly with respect to the additive character $\psi
\circ \tr_{E/F}$. We note that with these choices of measures we have
$$
\prod_{v < \infty} \vol(U_{F_v}, d^\times x_v) = |\Delta_F|^{-\frac 12},
$$
and similarly for $E$. We also have
$$
\vol(F_v^\times\bs E_v^\times) = \vol(\R^\times\bs\C^\times) = 2
$$
for $v\in\Sigma_\infty$.

For the group $D$ we recall that there exists an $\eps\in F^\times$, well defined in
$F^\times/NE^\times$, such that $D$ is isomorphic to
$$
\left\{ \bmx \alpha&\eps \beta\\\bar{\beta}&\bar{\alpha} \emx : \alpha, \beta \in E \right\}.
$$
For a place $v$ of $F$ we take the Tamagawa measure $dg_v$ on $D^\times_v$ which is given by
$$
dg_v = L(1, 1_{F_v}) |\eps|_v \frac{d\alpha_v \ d\beta_v}{|\alpha_v\bar{\alpha}_v - \eps \beta_v
\bar{\beta}_v|_v^2}.
$$
We note that this measure only depends on the choice of $\eps$ modulo $NE^\times$. Furthermore with this definition
$$
\vol(R^\times_v) = L(2, 1_{F_v})^{-1}  \vol(U_{F_v}, d^\times x_v)^4
$$
for non-archimedean $v \nmid \N$ since the isomorphism of  $D_v^\times$ with $\GL(2,F_v)$ preserves Tamagawa measures. We have
$$
\vol(R^\times_v)  = \frac{L(2, 1_{F_v})^{-1}}{q_v - 1} \vol(U_{F_v}, d^\times x_v)^4
$$
for $v \mid \N$ by a straightforward calculation. We also note that
$$
\vol(G(F_v)) = 4 \pi^2
$$
for $v \in \Sigma_\infty$.

Globally we take the product of these local measures and give discrete subgroups the counting
measures. In this way we get
$$
\vol(\A_F^\times E^\times \bs \A_E^\times) = 2 L(1, \eta)
$$
and
$$
\vol(G(F) \bs G(\A_F)) = 2.
$$

\section{Spectral side of the trace formula}
\label{spectral}

Having fixed the quaternion algebra $D$ we have taken $G$ to be the algebraic group defined over $F$ with $G(F) = D^\times/F^\times$. Since $D$ is anisotropic the quotient $G(F)\bs G(\A_F)$ is compact and hence, as a representation of $G(\A_F)$,
$$
L^2(G(F)\bs G(\A_F)) = \widehat{\bigoplus}_{\pi'\in\mathcal A(G)} V_{\pi'}.
$$
Here the sum is taken over the irreducible automorphic representations $\mathcal A(G)$ of $G(\A_F)$ and for each $\pi'\in\mathcal A(G)$, $V_{\pi'}$ denotes the space of $\pi'$. The Jacquet-Langlands correspondence \cite{jacquet:1970} yields an injection $\JL : \mathcal A(G) \hookrightarrow \mathcal A(\PGL(2))$, where $\mathcal A(\PGL(2))$ denotes the set of automorphic representations of $\PGL(2, \A_F)$. The set $\calA(G)$ can be decomposed as,
$$
\calA(G) = \calA_{cusp}(G) \amalg \calA_{res}(G),
$$
where
$$
\calA_{cusp}(G) = \left\{ \pi'\in\calA(G) : \JL(\pi') \text{ is cuspidal}\right\}
$$
and
$$
\calA_{res}(G) = \left\{ \delta \circ N_D : \delta : F^\times \bs \A_F^\times \to \{\pm 1\}\right\}.
$$

The compatibility between the local and global Jacquet-Langlands correspondence yields the following.

\begin{fact}
The image of $\calA_{cusp}(G)$ under $\JL$ is equal to the set of cuspidal automorphic representations $\pi = \otimes_v \pi_v$ of $\PGL(2, \A_F)$ such that $\pi_v$ is a discrete series representation of $\PGL(2, F_v)$ at all places $v$ where $D$ ramifies. In particular $\F(\N, 2\k)$ is contained in the image of $\JL$ and for $\pi' = \otimes_v \pi'_v \in\calA_{cusp}(G)$ we have $\JL(\pi')\in\F(\N, 2\k)$ if and only if,
\begin{enumerate}
  \item for $v\mid\N$, $\pi'_v \cong \delta_v \circ N_{D_v}$ where $\delta_v : F_v^\times \to \C^\times$ is an unramified character (necessarily of order at most 2),
  \item for $v\in\Sigma_\infty$, $\pi'_v \cong \Sym^{2k_v-2} V \otimes \det^{1 - k_v}$, where $V$ denotes the irreducible 2-dimensional representation of $D^\times(F_v)$ coming from the isomorphism $D^\times(\overline{F}_v) \iso \GL(2, \C)$, and
  \item for all other $v$, $\pi'_v$ is unramified.
\end{enumerate}
We set $\F'(\N, 2\k) = \{ \pi' \in \calA_{cusp}(G) : \JL(\pi') \in \F(\N, 2\k)\}$.
\label{factJL}
\end{fact}

Let $f\in C^\infty_c(G(\A_F))$, integrating $f$ against the action of $G(\A_F)$ gives a linear map
$$
R(f) : L^2(G(F)\bs G(\A_F)) \to L^2(G(F)\bs G(\A_F))
$$
defined by
$$
(R(f)\phi)(x) = \int_{G(\A_F)} f(g) \phi(xg) \ dg.
$$
From the spectral decomposition of $L^2(G(F)\bs G(\A_F))$ one sees that $R(f)$ is an integral operator with kernel,
$$
K_f(x, y) = \sum_{\pi'\in\calA(G)} \sum_{\phi\in\B(\pi')} (R(f)\phi)(x) \overline{\phi(y)},
$$
where for each $\pi'\in\calA(G)$, $\B(\pi')$ denotes an orthonormal basis of $V_{\pi'}$.

Having fixed the embedding $E \hookrightarrow D$ we get an injection $\A_E^\times \hookrightarrow D^\times(\A_F)$. We define a distribution
$$
I(f) = \int_{\A_F^\times E^\times \bs \A_E^\times} \int_{\A_F^\times E^\times \bs \A_E^\times}
K_f(t_1, t_2) \Omega(t_1^{-1} t_2) \ dt_1 \ dt_2.
$$
The spectral expansion for the kernel $K_f(x, y)$ gives,
$$
I(f) = \sum_{\pi'\in\mathcal A(G)} \sum_{\phi\in\B(\pi')} \int_{\A_F^\times E^\times \bs \A_E^\times} (R(f)\phi)(t_1)
\Omega^{-1}(t_1) \ dt_1 \overline{\int_{\A_F^\times E^\times \bs \A_E^\times} \phi(t_2)
\Omega^{-1}(t_2) \ dt_2}.
$$
For each $\pi'\in\calA(G)$ we define
$$
I_{\pi'}(f) = \sum_{\phi\in\B(\pi')} \int_{\A_F^\times E^\times \bs \A_E^\times} (R(f)\phi)(t_1)
\Omega^{-1}(t_1) \ dt_1 \overline{\int_{\A_F^\times E^\times \bs \A_E^\times} \phi(t_2)
\Omega^{-1}(t_2) \ dt_2}.
$$
We set
$$
I_{cusp}(f) = \sum_{\pi'\in\calA_{cusp}(G)} I_{\pi'}(f)
$$
and
$$
I_{res}(f) = \sum_{\pi'\in\calA_{res}(G)} I_{\pi'}(f)
$$
so that $I(f) = I_{cusp}(f) + I_{res}(f)$. In the same we also write
$$
K_f(x, y) = K_{f, cusp}(x, y) + K_{f, res}(x, y).
$$

Our goal in this section is to choose a suitable test function $f\in C^\infty_c(G(\A_F))$ such that $R(f)$ kills all $\pi'\in\calA_{cusp}(G)$ such that $\JL(\pi') \not\in \F(\N, 2\k)$. We will then compute $I(f)$ and use a precise version of Waldspurger's formula \cite[Theorem 4.1]{me:periods} to relate $I_{cusp}(f)$ to the $L$-values being considered in this paper.

\subsection{The test function} \label{testfunction}

We fix a finite prime $\p \nmid \N$ of $F$. We will consider test functions
$$
f = \prod_v f_v \in C^\infty_c(G(\A_F))
$$
defined as follows.

At a finite place $v \neq \p$ we take $f_v=\1_v$, the characteristic function of $Z_vR_v^\times$. At $v = \p$ we allow $f_\p$ to be any
element in the Hecke algebra $\H(G(F_\p), Z_\p R^\times_\p)$.

Let $v\in\Sigma_\infty$. We fix an isomorphism $D^\times (\overline{F_v}) \cong \GL(2, \C)$ which gives an irreducible $2$-dimensional representation $V$ of $D^\times (F_v)$. We set $\pi'_{2k_v} = \Sym^{2k_v-2} V \otimes \det^{-k_v + 1}$, which descends to a well defined representation of $G(F_v)$ corresponding, via the local Jacquet-Langlands correspondence, to the weight $2k_v$ discrete series on $\PGL(2,\R)$. We note that $\dim \pi'_{2k_v} = 2k_v - 1$. Let $\langle \ , \ \rangle$ be a $G(F_v)$ invariant inner product of $\pi'_{2k_v}$. As is well known, since $|m_v| < k_v$ the subspace of $\pi'_{2k_v}$,
$$
\pi'_{2k_v}(\Omega_v) = \left\{ w \in \pi'_{2k_v} : \pi'_{2k_v}(t)w = \Omega_v(t) w \text{ for all } t\in E_v^\times \right\}
$$
is 1-dimensional. We fix a unit vector $w_v \in \pi'_{2k_v}(\Omega_v)$ and define $f_v\in C^\infty_c(G(F_v))$ by
$$
f_v(g) = \overline{\langle \pi'_{2k_v}(g) w_v, w_v\rangle}.
$$

Given $f_\p\in\H(G(F_\p), Z_\p R^\times_\p)$ we denote by $\hat{f}_\p$ the function on unramified representations $\pi_\p$ defined by
$$
\hat{f}_\p(\pi_\p) = \frac{1}{\vol(Z_\p\bs Z_\p R_\p^\times)} \Tr \pi_\p (f_\p).
$$
We also view $\hat{f}_\p$ as a function on $[-2, 2]$ in the usual way.

\subsection{Local preliminaries}
\label{localprelims}

Before continuing onto the calculation of $I(f)$ for $f$ as in Section \ref{testfunction} we record some local preliminaries. For an irreducible admissible representation $\sigma$ of $G(F_v)$ acting on the space $V_{\sigma}$ and $f_v\in C^\infty_c(G(F_v))$ we let,
$$
\sigma(f_v) : V_{\sigma} \to V_{\sigma} : w \mapsto \int_{G(F_v)} f_v(g_v) \sigma(g_v)w \ dg_v.
$$

We record the following basic results for use later.

\begin{lemma}
Let $v$ be a finite place of $F$ not dividing $\N$ and let $f_v\in C^\infty_c(G(F_v))$ be as in Section \ref{testfunction}. Let $\sigma$ be an irreducible unitarizable representation of $G(F_v)$. Then $\sigma(f_v)$ kills the orthogonal complement of $\sigma^{R_v^\times}$ in $V_\sigma$ and for $w \in \sigma^{R_v^\times}$,
$$
\sigma(f_v) w = \vol(Z_v\bs Z_v R_v^\times) \hat{f}_v(\sigma) w.
$$
Furthermore $\sigma^{R_v^\times}$ has dimension at most one.
\end{lemma}

\begin{proof}
For $v\nmid\N$, $G(F_v) \cong \PGL(2, F_v)$ and this result is well known.
\end{proof}

\begin{lemma}
Let $v$ be a finite place of $F$ dividing $\N$ and let $f_v\in C^\infty_c(G(F_v))$ be as in Section \ref{testfunction}. Let $\sigma$ be an irreducible unitarizable representation of $G(F_v)$. Then $\sigma(f_v)$ kills the orthogonal complement of $\sigma^{R_v^\times}$ in $V_\sigma$ and for $w \in \sigma^{R_v^\times}$,
$$
\sigma(f_v) w = \vol(Z_v\bs Z_v R_v^\times) w.
$$
Furthermore if $\sigma^{R_v^\times} \neq 0$ then $\sigma = \delta \circ N_{D_v}$ for an unramified character $\delta$ of $F_v^\times$.
\end{lemma}

\begin{proof}
The first part of the lemma is clear. It remains to prove the last statement. In this case we use the exact sequence,
$$
\xymatrix{1 \ar[r] & R_v^\times \ar[r] & D_v^\times \ar[r]^{N_{D_v}} & F^\times_v/U_{F_v} \ar[r] & 1}.
$$
From which it follows that if $\sigma^{R_v^\times} \neq 0$ then $\sigma = \delta \circ N_{D_v}$ for a character $\delta$ of $F^\times_v/U_{F_v}$.
\end{proof}

\begin{lemma}
Let $v$ be an infinite place of $F$ and let $f_v\in C^\infty_c(G(F_v))$ be as in Section \ref{testfunction}. Let $\sigma$ be an irreducible unitarizable representation of $G(F_v)$. Then $\sigma(f_v)$ kills the space of $\sigma$ unless $\sigma \cong \pi'_{2k_v}$. Furthermore for $\sigma = \pi'_{2k_v}$ the map $\pi'_{2k_v}(f_v)$ kills the orthogonal complement of $\pi'_{2k_v}(\Omega_v)$ and for $w \in \pi'_{2k_v}(\Omega_v)$,
$$
\pi'_{2k_v}(f_v) w = \frac{\vol(G(F_v))}{2k_v - 1} w.
$$
\end{lemma}

\begin{proof}
Since $f_v$ is a matrix coefficient of $\pi'_{2k_v}$ it follows that $\sigma(f_v)$ kills the space of $\sigma$ unless $\sigma \cong \pi'_{2k_v}$. On the other hand if $0 \neq w\in\pi'_{2k_v}(\Omega_v)$ then
$$
f_v(g) = \frac{\overline{\langle \pi'_{2k_v}(g)w, w \rangle}}{\langle w, w\rangle}.
$$
Since $\pi'_{2k_v}(\Omega_v)$ is spanned by $w$ it follows that $\pi'_{2k_v}(f_v)$ kills the orthogonal complement of $\pi'_{2k_v}(\Omega_v)$ in $\pi'_{2k_v}$.
Finally,
\begin{align*}
\langle \pi'_{2k_v}(f_v) w, w\rangle &= \int_{G(F_v)} f_v(g) \langle \pi'_{2k_v}(g)w, w\rangle \ dg\\
&= \int_{G(F_v)} \frac{|\langle \pi'_{2k_v}(g)w, w\rangle|^2}{\langle w, w\rangle} \ dg\\
&= \frac{\vol(G(F_v))}{\dim \pi'_{2k_v}} \langle w, w\rangle.
\end{align*}
The last part of the Lemma now follows.
\end{proof}

As a temporary expedient we set,
\begin{align*}
\lambda(\N, \k) &= \prod_{v < \infty} \vol(Z_v \bs R_v^\times Z_v) \prod_{v\in\Sigma_\infty}
\frac{\vol(G(F_v))}{2k_v-1}\\
&= \frac{1}{{|\Delta_F|}^{\frac{3}{2}}L(2, 1_F)} \prod_{v \mid \N} \frac{1}{q_v - 1}
\prod_{v\in\Sigma_\infty} \frac{4 \pi}{2k_v-1}.
\end{align*}

\subsection{Global calculations}

Having chosen $f\in C^\infty_c(G(\A_F))$ in Section \ref{testfunction} we now set about computing
$$
I(f) = I_{cusp}(f) + I_{res}(f).
$$
We begin with the calculation of $I_{res}(f)$ which, as we shall see, is often zero. We define $X^{un}(F)$ to be the set of unramified characters $\chi : F^\times \bs \A_F^\times \to \{\pm 1\}$.

\begin{lemma}
For $f\in C^\infty_c(G(\A_F))$ as in Section \ref{testfunction}, $K_{f, res}(x,y) \equiv 0$ unless $2k_v = 2$ for all $v\in\Sigma_{\infty}$, in which case
$$
K_{f, res}(x, y) = \frac{\lambda(\N, \k)}{\vol(G(F)\bs G(\A_F))} \sum_{\delta\in X^{un}(F)} \delta(N_D(x y^{-1}))
\hat{f}_\p(\delta_\p\circ N_D).
$$
\label{kerreslem}
\end{lemma}

\begin{proof} We recall that
$$
\calA_{res}(G) = \left\{ \delta \circ N_D : \delta : F^\times \bs \A_F^\times \to \{\pm 1\}\right\}.
$$
We fix a character $\delta : F^\times \bs \A_F^\times \to \{\pm 1\}$. By definition we have
\begin{align*}
(R(f)(\delta \circ N_D))(x) &= \int_{G(\A_F)} f(g) \delta(N_D(x g)) \ dg\\
&= \prod_v \int_{G(F_v)} f_v(g_v) \delta_v(N_{D_v}(x g_v)) \ dg_v.
\end{align*}
We now set about computing the local integrals. Suppose first that $v$ is a non-archimedean place. If $v \neq \p$ then $f_v$ is the characteristic function of $Z_v R_v^\times$. Hence,
$$
\int_{G(F_v)} f_v(g_v) \delta_v(N_{D_v}(x g_v)) \ dg_v = \delta_v(N_{D_v}(x)) \int_{Z_v\bs Z_v R_v^\times} \delta_v(N_{D_v}(g_v)) \ dg_v.
$$
Since the norm map $N_{D_v} : R_v^\times \to U_{F_v}$ is surjective,
$$
\int_{Z_v\bs Z_v R_v^\times} \delta_v(N_{D_v}(x g_v)) \ dg_v =
\left\{
  \begin{array}{ll}
    0, & \hbox{if $\delta_v$ is ramified;} \\
    \delta_v(N_{D_v}(x)) \vol(Z_v\bs Z_v R_v^\times), & \hbox{if $\delta_v$ is unramified.}
  \end{array}
\right.
$$
If $v = \p$, then
\begin{align*}
\int_{G(F_v)} f_\p(g_\p) \delta_\p(N_{D_\p}(x g_\p)) \ dg_\p &= \delta_\p(N_{D_\p}(x)) \int_{G(F_\p)} f_\p(g_\p) \delta_\p(N_{D_\p}(g_\p)) \ dg_\p\\
&= \delta_\p(N_{D_\p}(x)) \Tr(\delta_\p\circ N_{D_\p})(f_\p).
\end{align*}
Clearly, since $f_\p$ is bi-$R_\p^\times$-invariant we have $\Tr(\delta_\p\circ N_{D_\p})(f_\p) = 0$ unless $\delta_\p$ is unramified, in which case
$$
\int_{G(F_v)} f_\p(g_\p) \delta_\p(N_{D_\p}(x g_\p)) \ dg_\p = \delta_\p(N_{D_\p}(x)) \vol(Z_\p R_\p^\times \bs R_\p^\times) \hat{f_\p}(\delta_\p\circ N_{D_\p}).
$$

Finally, let $v$ be an archimedean place of $F$. In this case we have $F_v \cong \R$ and, under this isomorphism, $N_{D_v}(D_v^\times) = \R^+$. Hence since $\delta_v$ is quadratic $\delta_v(N_{D_v}(g_v)) = 1$ for all $g_v \in D_v^\times$. Thus for $v\in\Sigma_\infty$,
\begin{align*}
\int_{G(F_v)} f_v(g_v) \delta_v(N_{D_v}(x g_v)) \ dg_v = \int_{G(F_v)} \overline{\langle \pi'_{2k_v}(g) w_v, w_v\rangle} \ dg_v,
\end{align*}
by definition of $f_v$. Clearly this integral is zero unless $\pi'_{2k_v}$ is the trivial representation which is the case if and only if $2k_v = 2$. Thus for $v\in\Sigma_\infty$,
\begin{align*}
\int_{G(F_v)} f_v(g_v) \delta_v(N_{D_v}(x g_v)) \ dg_v =
\left\{
  \begin{array}{ll}
    0, & \hbox{if $2k_v > 2$;} \\
    \vol(G(F_v)), & \hbox{if $2k_v = 2$.}
  \end{array}
\right.
\end{align*}

Putting these local calculations together shows that $(R(f)(\delta \circ N_D))(x)$ is zero unless $\delta$ is everywhere unramified and $k_v = 1$ for all $v\in\Sigma_\infty$ in which case
$$
(R(f)(\delta \circ N_D))(x) = \delta(N_D(x)) \hat{f_\p}(\delta_\p\circ N_{D_\p}) \prod_{v < \infty} \vol(Z_v \bs Z_v R_v^\times) \prod_{v\in\Sigma_\infty} \vol(G(F_v)).
$$

Finally to finish the proof it remains to observe that
$$
K_{f, res}(x, y) = \sum_{\delta} \frac{(R(f)(\delta \circ N_D))(x) \overline{(\delta \circ N_D)(y)}}{\vol(G(F)\bs G(\A_F))}.
$$
\end{proof}

We can now compute $I_{res}(f)$.

\begin{lemma}
For $f\in C^\infty_c(G(\A_F))$ as in Section \ref{testfunction},
$$
I_{res}(f) = C(\k, \Omega, f_\p) \lambda(\N, \k) \frac{\vol(\A_F^\times E^\times \bs
\A_E^\times)^2}{\vol(G(F)\bs G(\A_F))},
$$
where $C(\k, \Omega, f_\p) = 0$ unless $k_v = 1$ for all $v\in\Sigma_\infty$ and $\Omega$ is of the
form $\Omega = \delta\circ N_{E/F}$ for an everywhere unramified character $\delta$ of
$F^\times\bs\A_F^\times$ of order at most 2. In this latter case we have
$$
C(\k, \Omega, f_\p) = \left\{
  \begin{array}{ll}
    \hat{f}_\p(\delta_\p\circ N_D) + \hat{f}_\p(\eta_\p\delta_\p\circ N_D), & \hbox{if $E/F$ is unramified everywhere;} \\
    \hat{f}_\p(\delta_\p\circ N_D), & \hbox{otherwise.}
  \end{array}
\right.
$$
\label{lemcont}
\end{lemma}

\begin{proof}
By definition
$$
I_{res}(f) = \int_{E^\times \A_F^\times \bs \A_E^\times} \int_{E^\times \A_F^\times \bs \A_E^\times} K_{f, res}(t_1, t_2) \Omega(t_1^{-1} t_2) \ dt_1 \ dt_2.
$$
From Lemma \ref{kerreslem} we see that $I_{res}(f) = 0$ unless $k_v = 1$ for all $v\in\Sigma_\infty$, in which case,
$$
I_{res}(f) = \frac{\lambda(\N, \k)}{\vol(G(F)\bs G(\A_F))} \sum_{\delta \in X^{un}(F)} \hat{f}_\p(\delta_\p \circ N_{D_\p}) \left| \int_{E^\times \A_F^\times \bs \A_E^\times} \delta(N_D(t)) \Omega(t^{-1}) \ dt\right|^2.
$$
Considering $E^\times$ as a subgroup of $D^\times$ we note that $N_D|_{E^\times} : E^\times \to F^\times$ is equal to $N_{E/F} : E^\times \to F^\times$. Hence,
$$
\left| \int_{E^\times \A_F^\times \bs \A_E^\times} \delta(N_D(t)) \Omega(t^{-1}) \ dt\right|^2 =
\left\{
  \begin{array}{ll}
    \vol(E^\times \A_F^\times \bs \A_E^\times)^2, & \hbox{if $\Omega = \delta \circ N_{E/F}$;} \\
    0, & \hbox{otherwise.}
  \end{array}
\right.
$$
Finally it suffices to note that if $\Omega = \delta \circ N_{E/F}$ with $\delta \in X^{un}(F)$ then the only other quadratic character $\chi$ such that $\Omega = \chi \circ N_{E/F}$ is $\chi = \delta \eta$, and since $\delta$ is assumed to be unramified, $\delta \eta\in X^{un}(F)$ if and only if $\eta$ is unramified.
\end{proof}

We now set about computing
$$
I_{cusp}(f) = \sum_{\pi'\in\calA_{cusp}(G)} I_{\pi'}(f).
$$
It is clear from the results of Section \ref{localprelims} and Fact \ref{factJL} that if $\pi'\in\calA_{cusp}(G)$ then $R(f)$ is zero on the space of $\pi'$ unless $\pi'\in\F'(\N,2\k)$. Hence,

\begin{lemma}
For $f\in C^\infty_c(G(\A_F))$ as in Section \ref{testfunction},
$$
I_{cusp}(f) = \sum_{\pi'\in\F'(\N, 2\k)} I_{\pi'}(f).
$$
\label{lemcuspspec}
\end{lemma}

It remains to compute $I_{\pi'}(f)$ for $\pi'\in\F'(\N, 2\k)$.

\begin{lemma}
For $f\in C^\infty_c(G(\A_F))$ as in Section \ref{testfunction} and $\pi'\in\F'(\N, 2\k)$,
$$
I_{\pi'}(f) = \lambda(\N, \k) \frac{L(1/2, \pi_E\otimes\Omega)}{L(1, \pi, Ad)}  \frac{L(2,1_F) L_{S(\Omega)}(1, \eta)^{2}\sqrt{|\Delta_F|}}{2
\sqrt{c(\Omega) |\Delta_E|}}    \prod_{v \mid \N} (1 -
q_v^{-1})  \prod_{v\in\Sigma_\infty} \frac{ \Gamma(2k_v)}{\pi
\Gamma(k_v+m_v)\Gamma(k_v-m_v)},
$$
where $\pi = \JL(\pi')$.
\label{lemperiodLvalue}
\end{lemma}

\begin{proof}
Let $\pi'\in\F'(\N,2\k)$ and let $V_{\pi'}$ denote the space of $\pi'$. Let $W_{\pi'}$ denote the space of $w\in V_{\pi'}$ such that,
\begin{enumerate}
  \item $\pi'(k)w = w$ for all $k\in\prod_{v < \infty} R_v^\times$, and
  \item $\pi'(t)w = \Omega(t)w$ for all $t\in E_v^\times$, $v\in\Sigma_\infty$.
\end{enumerate}
By the results of Section \ref{localprelims} $W_{\pi'}$ is one dimensional. We fix a non-zero element $\phi_{\pi'} \in W_{\pi'}$. By our choice of $f$ it again follows from the results of Section \ref{localprelims} that $R(f)$ kills the orthogonal complement of $\phi_{\pi'}$ in $\pi'$ and
$$
R(f)\phi_{\pi'} = \lambda(\N, \k) \hat{f}_\p(\pi_\p) \phi_{\pi'}.
$$
Hence,
$$
I_{\pi'}(f) = \lambda(\N, \k) \hat{f}_\p(\pi_\p) \frac{\left|\int_{\A_F^\times E^\times \bs
\A_E^\times} \phi_{\pi'}(t) \Omega^{-1}(t) \ dt \right|^2}{(\phi_{\pi'}, \phi_{\pi'})}.
$$
We can now apply \cite[Theorem 4.1]{me:periods} which gives, for $\pi'\in\F'(\N, 2\k)$, an equality between
$$
\frac{\left|\int_{\A_F^\times E^\times \bs \A_E^\times} \phi_{\pi'}(t) \Omega^{-1}(t) \ dt
\right|^2}{(\phi_{\pi'}, \phi_{\pi'})}
$$
and
$$
\frac{L(1/2, \pi_E\otimes\Omega)}{L(1, \pi, Ad)}  \frac{L(2,1_F) L_{S(\Omega)}(1, \eta)^{2}\sqrt{|\Delta_F|}}{2 \sqrt{c(\Omega) |\Delta_E|}} \prod_{v \mid \N} (1 - q_v^{-1})  \prod_{v\in\Sigma_\infty} \frac{ \Gamma(2k_v)}{\pi \Gamma(k_v+m_v)\Gamma(k_v-m_v)},
$$
where $S(\Omega)$ is the set of places of $F$ above which $\Omega$ is ramified and $L(s, \pi, Ad)$ is the adjoint $L$-function of $\pi$. For the ease of the reader we note that for $\pi = \JL(\pi')\in\F(\N, 2\k)$ we have, in the notation of \cite[Theorem 4.1]{me:periods}, $S'(\pi) = \emptyset$; at the places $v\in \Ram(\pi)$, $e_v(E/F) = 1$; and for $v\in\Sigma_\infty$, $C_v(E, \pi, \Omega)$ is equal to the quotient of gamma functions appearing above. The $L(2, 1_F)$ term appears in the formula here, and not in \cite[Theorem 4.1]{me:periods}, due to a difference in the choice of measures on $G(\A_F)$.
\end{proof}

We recall that
$$
\vol(G(F)\bs G(\A_F)) = 2,
$$
and
$$
\vol(\A_F^\times E^\times \bs \A_E^\times) = 2 L(1, \eta),
$$
and that we defined
$$
\lambda(\N, \k) = \frac{1}{{|\Delta_F|}^{\frac{3}{2}}L(2, 1_F)} \prod_{v \mid \N} \frac{1}{q_v - 1}
\prod_{v\in\Sigma_\infty} \frac{4 \pi}{2k_v-1}.
$$
Now by combining our calculation of $I_{res}(f)$ from Lemma \ref{lemcont} with our calculation of $I_{cusp}(f)$ from Lemmas \ref{lemcuspspec} and \ref{lemperiodLvalue} we obtain the following.

\begin{proposition}
For $f$ as in Section \ref{testfunction}, $I(f)$ is equal to the sum of
$$
\frac{L_{S(\Omega)}(1, \eta)^{2}}{2 {|\Delta_F|}^2 \sqrt{c(\Omega) d_{E/F}}} \frac{4^{[F:\Q]}}{|\N|}
\binom{2\k - 2}{\k + \m - 1}  \sum_{\pi\in\F(\N, 2\k)}
\frac{L(1/2, \pi_E\otimes\Omega)}{L(1, \pi, Ad)} \hat{f}_\p(\pi_\p),
$$
and
$$
C(\k, \Omega, f_\p)  \frac{2 L(1, \eta)^2 }{{|\Delta_F|}^{\frac{3}{2}}L(2,1_F)}  \prod_{v
\mid \N} \frac{1}{q_v - 1}  \prod_{v\in\Sigma_\infty} \frac{4 \pi}{2k_v-1},
$$
where $C(\k, \Omega, f_\p)$ is defined  in Lemma \ref{lemcont}.
\label{propspectral}
\end{proposition}

\section{Geometric side of the trace formula}
\label{geometric}

Recall that
$$
I(f) = \int_{\A_F^\times E^\times \bs \A_E^\times} \int_{\A_F^\times E^\times \bs \A_E^\times}
K_f(t_1, t_2) \Omega(t_1^{-1} t_2) \ dt_1 \ dt_2.
$$
A quick calculation shows that the following matrices are a set of representatives for the double
cosets of $E^\times \backslash D^\times / E^\times$:
\[
\left \{ \begin{pmatrix} 1& 0 \\ 0&1 \end{pmatrix} , \begin{pmatrix} 0& \epsilon \\ 1& 0
\end{pmatrix}\right\} \cup \left \{ \begin{pmatrix} 1 & \epsilon x \\ \overline x &1 \end{pmatrix}
: x\in E^\times/E^1\right \}.
\]
The first two cosets are referred to as \textit{irregular cosets} and the remaining cosets are
called \textit{regular cosets}.

By the geometric expansion of the relative trace formula \cite[Section 1 (8)]{jacquet:2001},
\[
I(f)=\sum_{\xi \in \epsilon NE^\times}I(\xi,f) + \vol(\A_F^\times E^\times \backslash
\A_E^\times)\left[I\left(0, f \right)+\delta (\Omega^2)I\left( \infty,f \right)\right]
\]
where for an idele class character $\chi$, $\delta(\chi)=1$ if $\chi=1$ and $\delta(\chi)=0$
otherwise.  For $\xi=\epsilon x \overline x$,
\[
I\left( \xi, f \right):=\int_{\A_F^\times  \bs \A_E^\times} \int_{\A_F^\times  \bs \A_E^\times}
f\left(t_1 \begin{pmatrix} 1 & \epsilon x \\ \overline x &1 \end{pmatrix}
t_2\right)\Omega(t_1t_2)dt_1dt_2,
\]
\[
I\left( 0, f \right):=\int_{\A_F^\times \bs \A_E^\times} f\left(t\right)\Omega(t)dt,
\]
and
\[
I\left( \infty, f \right):= \int_{\A_F^\times  \bs \A_E^\times} f\left(t \begin{pmatrix}0 & \eps
\\ 1 & 0 \end{pmatrix}\right)\Omega(t)dt.
\]
We remark that these integrals factor as a product of local ones, which we denote with similar notation.
We note that the geometric expansion in \cite{jacquet:2001} has  $\delta(\Omega^2)$ next
to $I(0,f)$ rather than $I(\infty, f)$. This is because  $K_f(t_1, t_2)$ is integrated
against $\Omega(t_1t_2)^{-1}$ rather than against $\Omega(t_1^{-1}t_2)$.

In this section we compute the period integrals $I(g, f)$ for $f$ chosen as in Section \ref{testfunction}. First we make the necessary local calculations in Sections \ref{irregularcosets} and \ref{regularcosets} and then we bring together the local calculations to write down the geometric side of the global relative trace formula in Section \ref{globalcalculations}. We note that the choice of local function $f_v$ is intrinsic to $G(F_v)$, as is the parameterization of the double cosets. Hence for simplicity we can fix the following identifications for the remainder of this section.

For $v \notin \Sigma_\infty$,
\begin{align*}
D_v=\left\{ \begin{pmatrix}  \alpha & \epsilon_v\beta \\ \overline \beta & \overline \alpha
\end{pmatrix} \right \}, E_v= \left\{ \begin{pmatrix} \alpha & 0 \\ 0& \overline \alpha
\end{pmatrix} \right \}, \text{ where }&\epsilon_v=1\text{ if } v \nmid \N,
\\
&\epsilon_v=\varpi_v \text{ if } v \mid \N,
\end{align*}
and for $v\in \Sigma_\infty$,
\[
D_v=\left\{ \begin{pmatrix}  \alpha & -\beta \\ \overline \beta & \overline \alpha \end{pmatrix}
\right \}, E_v= \left\{ \begin{pmatrix} \alpha & 0 \\ 0& \overline \alpha \end{pmatrix} \right \}.
\]

For $v$ that splits in $E$, let $E_v=F_v\oplus F_v$. For $v \notin \Sigma_\infty$, let $\tau_v$ be
such that $\O_{E_v}=\O_{F_v}[\tau_v]$ and $\tau_v$ is a uniformizer in $E_v$ if $E_v/F_v$ is
ramified. If $v$ is inert in $E$, let $\varpi_{E_v}=\varpi_v$. If $v$ ramifies, let
$\varpi_{E_v}=\tau_v$. If $v$ splits let $\varpi_{E_v}\in\O_{E_v}$ be such that $v(N(\varpi_{E_v}))
= 1$. For a valuation $v$ of $F$ that is not split in $E$, let $v_E$ be the corresponding valuation
on $E$. Let
\[
R_v=\left \{ \begin{pmatrix} \alpha & \epsilon_v\beta \\ \overline \beta & \overline \alpha
\end{pmatrix} : \alpha \in \frac{1}{(\overline \tau_v -\tau_v)\varpi_{v}^{n(\Omega_v)}} \left(
\O_{F_v} +\varpi_v^{n(\Omega_v)} \O_{E_v}\right), \alpha + \beta \in \O_{F_v} +
\varpi_v^{n(\Omega_v)} \O_{E_v} \right\}.
\]
This is a maximal order in $D_v$ such that $R_v\cap E_v= \O_{F_v}+\varpi_v^{n(\Omega_v)}\O_{E_v}$.

For the rest of this section we will drop the $v$ from our notation and let $n=n(\Omega_v)$ when it is clear that
we are working locally.

\subsection{Irregular cosets}
\label{irregularcosets}

First we compute the orbital integrals associated to the irregular cosets.
\begin{lemma} \label{lemN0}
Let $v$ divide $\N$. Then
\[
I(0, f_v)=\vol (F_v^\times \bs F_v^\times (1+\varpi_v^{n(\Omega_v)}\O_{E_v}))^\times.
\]
\end{lemma}
\begin{proof}
We have
\begin{align*}
I(0,f_v)&=\int_{F_v^\times \bs E_v^\times} f_v \left( \bmx a&0 \\0 & \overline a
\emx\right)\Omega_v(a)d^\times a
\\
&=\vol (F_v^\times \bs F_v^\times (1+\varpi_v^{n(\Omega_v)}\O_{E_v}))^\times.
\end{align*}
\end{proof}
\begin{lemma} \label{lemN}
Let $v$ divide $\N$. Then
\[
I\left( \infty, f_v\right)=0.
\]
\end{lemma}
\begin{proof}
\[
I\left(\infty, f_v \right)=\int_{F_v^\times\backslash E_v^\times} f_v\left(\begin{pmatrix}
0&\varpi_v a \\ \overline a & 0\end{pmatrix}\right)\Omega_v(a)d^\times a.
\]
Clearly
\[
\begin{pmatrix}
0&\varpi_v a \\ \overline a & 0\end{pmatrix} \notin Z_vR_v^\times
\]
because $Z_vR_v^\times$ only contains matrices, $g$, where $v(\det g)$ is even.
\end{proof}

Before we compute the irregular orbital integrals for finite $v$ away from $\N$, we need the following
technical lemma.
\begin{lemma}
Let $v$ be a finite prime not dividing $\N$. Let $\alpha = a_\alpha +
b_\alpha\tau_v$ and $\beta = a_\beta + b_\beta\tau_v$. Then $F_v^\times R_v^\times \alpha
R_v^\times = F_v^\times R_v^\times \beta R_v^\times$ if and only if
$$
v(N(\alpha)) - 2 \min \{ v(a_\alpha), v(\varpi_v^{-n(\Omega_v)} b_\alpha) \} =   v(N(\beta)) - 2
\min \{ v(a_\beta), v(\varpi_v^{-n(\Omega_v)} b_\beta) \} .
$$
\label{lemdoublecoset}
\end{lemma}

\begin{proof}
For the proof we can take $D = M(2, F)$ and embed $E$ as
$$
E \hookrightarrow M(2, F) : a + b\tau \mapsto \bmx a+b\Tr_{E/F}(\tau)&b\varpi^{-n}\\-b\varpi^n
N(\tau)&a \emx.
$$
We set $R = M(2, \O_F)$, then we have $R\cap E = \O_F + \varpi^n \O_E$. We recall that
$$
\GL(2, F) = \bigsqcup_{m\geq 0} F^\times R^\times \bmx \varpi^m&\\&1 \emx R^\times,
$$
and
$$
g = \bmx a&b\\c&d \emx \in F^\times R^\times \bmx \varpi^m&\\&1 \emx R^\times,
$$
if and only if
$$
m =   v(\det g) - 2 \min \{ v(a), v(b), v(c), v(d) \} .
$$
Thus we see that
$$
\alpha = a + b\tau \in F^\times R^\times \bmx \varpi^m&\\&1 \emx R^\times,
$$
if and only if
$$
m =  v(N(\alpha)) - 2 \min \{ v(a), v(\varpi^{-n} b) \}.
$$
\end{proof}

For $f_v\in \H(G(F_v), Z_vR_v^\times)$ we have
\begin{align*}
I(0, f_v) &= \int_{F_v^\times\bs E_v^\times} f_v(a) \Omega_v(a) \ d^\times a\\
& = \vol(F_v^\times \bs F_v^\times (\O_{F_v} + \varpi_v^{n(\Omega_v)} \O_{E_v})^\times) \I(f_v),
\end{align*}
where
\[
\I(f_v):=\sum_{\alpha\in F_v^\times(\O_{F_v} + \varpi_v^{n(\Omega_v)} \O_{E_v})^\times \bs
E_v^\times} f_v(\alpha) \Omega_v(\alpha).
\]
For finite $v$ away from $\N$,  $\begin{pmatrix} 0 & \epsilon \\ 1 & 0
\end{pmatrix}\in R_v^\times$, and hence
\begin{align*}
I(\infty, f_v)&=\int_{F_v^\times \bs E_v^\times} f_v\left( a\begin{pmatrix} 0& \epsilon_v \\ 1&0
\end{pmatrix}\right) \Omega_v(a)d^\times a
\\
&=\int_{F_v^\times \bs E_v^\times} f_v\left( a\right) \Omega_v(a)d^\times a
\\
&=I(0, f_v).
\end{align*}
In the following three lemmas, we compute $\I(f_v)$ for $v$ inert, ramified and split in $E$.
\begin{lemma} \label{lemIrUn}
Let $v$ be a finite prime not dividing $\N$, which  is
inert in $E$. For $f_v\in\H(G(F_v), Z_vR_v^\times)$,
\begin{align*}
\I( f_v) =
\begin{cases} f_v(1) & \text{ when } n(\Omega_v)=0
\\
f_v(1) - f_v(1 + \tau_v\varpi_v^{n(\Omega_v)-1})
&
\text{ when } n(\Omega_v) > 0.
\end{cases}
\end{align*}
\end{lemma}

\begin{proof}
Recall that $n=n(\Omega_v)$.
We note that a set of representatives for $F^\times (\O_F + \varpi^n \O_E)^\times \bs E^\times$ is
given by
$$
\left\{ 1 + b\tau : b\in \O_F/\p_F^n\right\} \cup \left\{ a + \tau : a\in \p_F/\p_F^n\right\}.
$$
Applying Lemma \ref{lemdoublecoset} we see that
$$
\I(f)=\sum_{k=0}^{n} f(1 + \tau \varpi^k) \sum_{b \in (\p_F^k/\p_F^n)^\times} \Omega(1 + b\tau) + f(\tau)
\sum_{a \in \p_F/\p_F^n} \Omega(a + \tau).
$$
The lemma follows by the observation that
$$
\sum_{b \in (\p_F^k/\p_F^n)^\times} \Omega(1 + b\tau) = \left\{
  \begin{array}{ll}
    0, & \hbox{if $0 \leq k \leq n - 1$;} \\
    -1, & \hbox{if $k = n - 1$;} \\
    1, & \hbox{if $k = n$,}
  \end{array}
\right.
$$
and that
$$
\sum_{a \in \p_F/\p_F^n} \Omega(a + \tau) = 0.
$$
\end{proof}

\begin{lemma} \label{lemIrRa}
Let $v$ be a finite prime which is ramified
in $E$. For $f_v\in\H(G(F_v), Z_vR_v^\times)$,
\begin{align*}
\I(f_v) =
\begin{cases}
f_v(1) + \Omega_v(\tau_v) f_v(\tau_v) & \text{ when } n(\Omega_v)=0
\\
f_v(1) - f_v(1 + \tau_v \varpi_v^{n(\Omega_v)-1}) &
\text{ when } n(\Omega_v) > 0.
\end{cases}
\end{align*}
\end{lemma}

\begin{proof}
We note that a set of representatives for $ F^\times (\O_F + \varpi^n \O_E)^\times\bs E^\times$ is
given by
$$
\left\{ 1 + b\tau : b\in \O_F/\p_F^n\right\} \cup \left\{ a\varpi + \tau : a\in
\O_F/\p_F^n\right\}.
$$
Applying Lemma \ref{lemdoublecoset} we see that
$$
\I(f)=\sum_{k=0}^{n} f(1 + \tau \varpi^k) \sum_{b \in (\p_F^k/\p_F^n)^\times} \Omega(1 + b\tau) +
f(\varpi + \tau) \sum_{a \in \O_F/\p_F^n} \Omega(a \varpi + \tau).
$$
The lemma follows by the observation that
$$
\sum_{b \in (\p_F^k/\p_F^n)^\times} \Omega(1 + b\tau) = \left\{
  \begin{array}{ll}
    0, & \hbox{if $0 \leq k \leq n - 1$;} \\
    -1, & \hbox{if $k = n - 1$;} \\
    1, & \hbox{if $k = n$.}
  \end{array}
\right.
$$
And that
$$
\sum_{a \in \O_F/\p_F^n} \Omega(a \varpi + \tau) = \left\{
  \begin{array}{ll}
    \Omega(\tau), & \hbox{if $n = 0$;} \\
    0, & \hbox{if $n > 0$.}
  \end{array}
\right.
$$
\end{proof}

\begin{lemma}\label{lemIrSp}
Let $v$ be a finite prime which is split in $E$.
Let $f_v\in\H(G(F_v), R_v^\times)$. When $n(\Omega_v)=0$ we have
$$
\I( f_v) =  \sum_{\alpha\in E_v^\times/F_v^\times \O_{E_v}^\times} \Omega_v(\alpha) f_v(\alpha),
$$
and when $n(\Omega_v) > 0$ we have
$$
\I( f_v) =  f_v(1) - f_v(1 + \tau_v\varpi_v^{n(\Omega_v)-1}) .
$$
\end{lemma}

\begin{proof}
In this case we take $R = M(2, \O_F)$ and embed $E$  as
$$
(a, b) \mapsto \bmx a&\varpi^{-n}(a-b)\\0&b \emx.
$$
By Lemma \ref{lemdoublecoset} we have
$$
(a, b) \in F^\times R^\times \bmx \varpi^m&0\\0&1 \emx R^\times
$$
if and only if $m = v(ab) - 2 \min\{v(a), v(\varpi^{-n}(a-b))\}$. We recall
\begin{align*}
\I(f) &= \sum_{\alpha\in F^\times(\O_{F} + \varpi^{n(\Omega)} \O_{E})^\times \bs
E^\times} f(\alpha) \Omega(\alpha)
\\
&= \sum_{\alpha \in U_F^n \bs F^\times} \Omega(\alpha, 1) f(\alpha, 1).
\end{align*}
Thus we see that if $n = 0$ then
$$
\I(f) =  f(1) + \sum_{m=1}^\infty (\Omega(\varpi^m,
1) + \Omega(\varpi^{-m}, 1)) f(\varpi^m,1) .
$$
On the other hand if $n > 0$, then only $\alpha \in U_F$ contribute to the sum, thus
\begin{align*}
\I(f) &=  \sum_{a\in U_F^n \bs U_F}
\Omega(a,1) f(a, 1)\\
&=  f(1,1) - f(1+\varpi^{n-1}, 1).
\end{align*}
\end{proof}

We now compute the irregular orbital integrals at the archimedean places.
\begin{lemma}
Let $v \in \Sigma_\infty$. Then
$
I\left(0,f_v \right)=  \vol(F_v^\times\backslash E_v^\times).
$
\label{lemI0}
\end{lemma}
\begin{proof}
We have
\begin{align*}
I\left( 0,f_v \right) &=  \int_{F_v^\times\backslash E_v^\times} f_v\left( \begin{pmatrix} a&0 \\
0& \overline a \end{pmatrix} \right)\Omega_v (a)d^\times a\\
&=  \vol(F_v^\times\backslash E_v^\times).
\end{align*}
\end{proof}

\begin{lemma}\label{lemarchfn}
Let $v\in\Sigma_\infty$. When $\alpha \neq 0$ we have
\begin{align*}
&f_v \bmx \alpha&-\beta\\\bar{\beta}&\bar{\alpha} \emx
\\
=
 &\frac{1}{(\alpha\bar{\alpha} +
\beta\bar{\beta})^{k_v-1}} \sum_{i=0}^{k_v - |m_v| - 1} (-1)^i \binom{k_v + m_v- 1}{i} \binom{k_v
- m_v - 1}{i} (\beta\bar{\beta})^{i} (\alpha\overline\alpha)^{k_v - 1 - i}
\left(\frac{\alpha}{\bar{\alpha}}\right)^{-m_v},
\end{align*}
 when $m_v=0$ we have
\begin{align*}
f_v \bmx 0&-\beta\\\bar{\beta}&0 \emx = (-1)^{k_v-1}
\end{align*}
and when  $m_v\neq 0$ we have
\[
f_v \bmx 0&-\beta\\\bar{\beta}&0 \emx =
 0.
\]
\end{lemma}

\begin{proof}
Let $\pi$ be the discrete series representation of $\PGL(2, \R)$ of weight $2k$. Recall that $\Omega$
is the character of $\C^\times$ given by
$$
\Omega : z\mapsto \left(\frac{z}{\bar{z}}\right)^m
$$
where $m$ is an integer with $|m| < k$. We have
$$
D^\times(\R) = \left\{ \begin{pmatrix}\alpha&-\beta\\\bar{\beta}&\bar{\alpha}\end{pmatrix} \in \GL(2,\C)
\right\}.
$$
Viewing $D^\times (\R) \subset \GL(2, \C)$ gives an irreducible 2-dimensional representation $V$ of $D^\times(\R)$.  We take $\pi'$ to be the representation of $G(\R)$ which corresponds to $\pi$ via the Jacquet-Langlands correspondence. Thus
$$
\pi' = \Sym^{2k-2}V \otimes \text{det}^{-(k-1)}.
$$
We realize $\pi'$ on the space of homogeneous polynomials in $x$ and $y$ of degree $2k-2$, with $g$ acting by
$$
\pi'(g) : x \mapsto \alpha x + \bar{\beta} y, y \mapsto -\beta x + \bar{\alpha} y.
$$
We set, for $0\leq i\leq 2k-2$,
$$
v_i = x^{i} y^{2k-2-i},
$$
so that
$$
\pi' \begin{pmatrix}\alpha&0\\0&\bar{\alpha}\end{pmatrix} v_i =
\left(\frac{\alpha}{\bar{\alpha}}\right)^{i - (k-1)} v_i.
$$
Hence we have
$$
\pi' \begin{pmatrix}\alpha&0\\0&\bar{\alpha}\end{pmatrix} v_{m + k - 1} = \Omega(\alpha) v_{m + k -
1}.
$$

We have, for $g\in G(\R)$,
$$
f(g) = \frac{\overline{\langle \pi'(g) v_{m + k - 1}, v_{m + k - 1} \rangle}}{\langle v_{m + k - 1}, v_{m + k
- 1} \rangle}.
$$
We compute that
\begin{align*}
& f \begin{pmatrix}\alpha&-\beta\\\bar{\beta}&\bar{\alpha}\end{pmatrix}
\\
 =
&\frac{1}{(\alpha\bar{\alpha} + \beta\bar{\beta})^{k-1}} \sum_{i=0}^{k-|m|-1} (-1)^i
\binom{k+m-1}{i} \binom{k-m-1}{i} (\beta\bar{\beta})^{i} \bar\alpha^{m + k - 1 - i} {\alpha}^{k - m
- 1 - i}
\end{align*}
and the lemma now follows.
\end{proof}

From the definition of $I(\infty, f_v)$ we have,
\begin{corollary} \label{corIrIn}
For $v\in\Sigma_\infty$,
$$
I(\infty, f_v) = \left\{
  \begin{array}{ll}
    \vol(F_v^\times\bs E_v^\times) (-1)^{k_v - 1}, & \hbox{if $m_v = 0$;} \\
    0, & \hbox{otherwise.}
  \end{array}
\right.
$$
\end{corollary}

\subsection{Regular cosets}
\label{regularcosets}

Recall that
\[
I(\xi, f_v)= \int_{(F_v^\times \bs E_v^\times)^2}f_v\left( \bmx ab& a \bar b x\epsilon_v \\ \bar a
b \bar x & \overline{ab} \emx \right) \Omega_v(ab) d^\times b d^\times a.
\]
By a change of variables this equals
\[
\int_{F_v^\times \bs E_v^\times}\int_{E_v^1} f_v \left( \bmx a & 0 \\ 0& \bar a \emx
\gamma_{xz}\right) \Omega_v(a) d^\times z d^\times a
\]
where $\gamma_{xz}= \bmx 1& xz \epsilon_v \\ \overline{xz} & 1 \emx$.
For $x\in E^\times_v$, we let
\[
I(x, E_v^1):= \left\{ z \in E_v^1: \exists a \in E_v^\times \text{ such that } a \gamma_{xz} \in
R_v^\times\right\}.
\]
In this section we compute $I(\xi, f_v)$ under certain simplifying assumptions, and  provide a
bound on $|I(\xi, f_v)|$ in all cases.

\subsubsection{Exact calculations}
We begin by calculating $I(\xi, f_v)$ when $v$ divides $\N$.
\begin{lemma} \label{lemReN}
Let $v$ divide $\N$. Then
\[
I(\xi, f_v) =
\begin{cases}
0 & \text{if } v(\xi)\leq 0
\\
\vol (F_v^\times \backslash E_v^\times)^2 & \text{if } v(\xi)\geq 1 .
\end{cases}
\]
\end{lemma}
\begin{proof}
It is easy to see that
$
a \gamma_{xz}
\in F^\times R^\times$ exactly when $v(x)\geq 0$. Hence $v(\xi)$ must be greater than or equal
to one.
\end{proof}

We now establish a vanishing result for $I(\xi, f_v)$ for finite $v$ away from $\N$.
\begin{lemma}\label{lemRe}
Let $v$ be a finite prime not dividing $\N$. Suppose that $f_v$ is the characteristic function of
the double coset  $F_v^\times R_v^\times \gamma R_v^\times$ for $\gamma \in R_v$. If
$v(1-\xi)>v(\d_{E/F} \mathfrak{c}(\Omega)\det(\gamma))$, then $I(\xi, f_v)=0$.
\end{lemma}
\begin{proof}
Clearly $a \gamma_{xz} \in F_v^\times R_v^\times \gamma R_v^\times$ if and only if there exists a $\lambda \in
F_v^\times$ such that $ \lambda a  \gamma_{xz} \in R_v^\times \gamma R_v^\times$. If
this matrix lies in $R_v^\times \gamma R_v^\times$, then $v(\lambda^2 a\overline a(1-\xi))=v(\det
(\gamma))$ and $\lambda a \in \frac{1}{(\overline \tau -\tau)\varpi_v^n}\O_E$. Thus the orbital
integral is zero unless  $v(1-\xi)\leq v(\d_{E/F} \mathfrak{c}(\Omega) \det (\gamma))$.
\end{proof}

In the next two lemmas we compute $I(\xi, \1_v)$ when $\Omega_v$ is unramified.
\begin{lemma}
Let $v$ be a finite prime not dividing $\N$, not split in $E$ and such that $n(\Omega_v)=0$. If $v$
ramifies in $E$, assume that the characteristic of the residue field of $F_v$ is  odd.  Then
$I(\xi, \1_v)$ equals
\begin{multline*}
\vol(F_v^\times \backslash F_v^\times U_{E_v}) \vol(F_v^\times \bs E_v^\times)
\Omega_v(\varpi_{E_v}^{\frac{v_E(1-\xi)}{2}}) \times
\begin{cases}
0 & \text{if $v(1-\xi)> v(\d_{E/F}  )$}
\\
1  & \text{if  $v(1-\xi)\leq 0$}
\\
\frac{1}{2} & \text{if $v(1-\xi)=v(\d_{E/F})>0$}.
\end{cases}
\end{multline*}
\label{lemRe1}
\end{lemma}
\begin{proof}
We look at our orbital integral,
\[
I(\xi, \1)=\int_{F^\times\backslash E^\times} \int_{E^1} \1 \left( a \gamma_{xz} \right) \Omega(a) d^\times z
d^\times a.
\]
By the definition of $R^\times$, $a \gamma_{xz}$ lies in $F^\times R^\times$ if and only if there is a
$\lambda \in F^\times$ such that
\begin{enumerate}
\item $v(\lambda^2a\overline a)=-v(1-\xi)$
\item $\lambda a \in \frac{1}{(\overline \tau -\tau)} \O_E$
\item $\lambda a(1+zx) \in \O_E$.
\end{enumerate}
If $v(1-\xi)\leq v(\d_{E/F})$, then
\begin{align*}
I(\xi, \1) &=\int_{\{z \in E^1: 1+xz \in \varpi_E^{\frac{v_E(1-\xi)}{2}}\O_E\}} \int_{
F^\times\backslash F^\times \varpi_E^{\frac{-v_E(1-\xi)}{2}}U_E}  \Omega(a) d^\times a d^\times z
\\
&=\vol\left(z \in E^1: 1+xz \in \varpi_E^{\frac{v_E(1-\xi)}{2}}\O_E \right)\vol(F^\times \bs
F^\times U_E) \Omega(\varpi_{E}^{\frac{v_E(1-\xi)}{2}}).
\end{align*}
The lemma now easily follows if $v(1-\xi)\leq 0$.

It now remains to calculate $\vol( \{ z\in E^1 : v_E(1 + xz) \geq 1 \} )$ when $v$ ramifies in $E$
and $v(1-\xi)=v(\d_{E/F})$. In this case we can take $\tau=\sqrt \varpi$. Since the volume only
depends on $x\bar{x}$ we can assume that $x \in U^1_E$. Hence we can write $x = x_1 + x_2
\sqrt{\varpi}$ with $x_1 \in U_F^1$ and $x_2 \in \O_F$. We write $z = \alpha \bar{\alpha}^{-1}$
with $\alpha = \alpha_1 + \alpha_2 \sqrt{\varpi}$. Then
\begin{align*}
1 + xz &= \bar{\alpha}^{-1} \left( \bar{\alpha} + x \alpha\right)\\
&= \bar{\alpha}^{-1} \left( \alpha_1 - \alpha_2 \sqrt{\varpi} + (x_1 + x_2 \sqrt{\varpi})(\alpha_1
+ \alpha_2 \sqrt{\varpi})\right)\\
&= \bar{\alpha}^{-1} \left( \alpha_1 (1 + x_1) + \alpha_2 x_2 \varpi + (\alpha_1 x_2 + \alpha_2
(x_1 - 1)) \sqrt{\varpi} \right).
\end{align*}
Hence we have $v_E(1 + xz) \geq 1$ if and only if $v_E(\alpha)$ is odd, from which we deduce that
$$
\vol( \{ z\in E^1 : v_E(1 + xz) \geq 1 \} ) = \frac 12 \vol(E^1).
$$
\end{proof}

\begin{lemma}
Let $v$ be a finite prime  which is  split in $E$ and such that $n(\Omega_v)=0$. Then
\begin{align*}
I(\xi, \1_v) = \left(\vol\left(U_{F_v} \right)\right)^2 \times
\begin{cases}
0 & \text{ if } v(1-\xi)>0
\\
(v(\xi)+1)
& \text{ if } v(1-\xi) = 0
\\
\Omega_v(\xi, 1)\sum_{l=0}^{|v(\xi)|} \Omega_v(\varpi_v^{2l}, 1) & \text{ if } v(1-\xi)<0.
\end{cases}
\end{align*}
\label{lemRe2}
\end{lemma}
\begin{proof}
Let $x=(x_1, x_2)$ and $k=v(1-\xi)$.
We look at our orbital integral,
\[
I(\xi, \1)=\int_{F^\times\backslash E^\times} \int_{E^1} \1 \left( \begin{pmatrix} a & a z x \\
\overline a \overline z \overline x & \overline a \end{pmatrix} \right) \Omega(a)
d^\times z d^\times a.
\]
A matrix in the integrand of this integral lies in $F^\times R^\times$ if and only if there is a
$\lambda \in F^\times$ such that
\begin{enumerate}
\item $v(\lambda^2a\overline a)=-v(1-\xi)$
\item $\lambda a \in \O_E$
\item $\lambda a(1+zx) \in \O_E$.
\end{enumerate}
It is easy to see from these conditions that if $v(1-\xi)\leq 0$,
\begin{align*}
I(\xi, \1)=\sum_{l=0}^{-k}\int_{ F^\times\backslash F^\times(\varpi^lU_F\times \varpi^{-k-l}U_F)}
\Omega(a) d^\times a \sum_{m=-v(x_1)-l}^{-k-l+v(x_2)} \int_{(\varpi^m, \varpi^{-m}) U_F}  d^\times
b.
\end{align*}
If $k=0$,
\[
I(\xi,\1)=(v(\xi)+1)\vol\left(U_F \right)^2 .
\]
If $k<0$, then $v(x_1x_2)=k$ and
\[
I(\xi, \1)=\sum_{l=0}^{-k} \Omega(\varpi^{k+2l},1) \vol\left(U_F \right)^2.
\]
\end{proof}

Finally, we compute the regular orbital integral for $v$ archimedean.
\begin{lemma}
For $v\in\Sigma_\infty$ and $\xi\in F_v$ with $\xi < 0$ we have
$$
I(\xi, f_v) = \frac{\vol(F_v^\times\bs E_v^\times)^2}{(1-\xi)^{k_v - 1}} \sum_{i=0}^{k_v - |m_v| -
1} \binom{k_v + m_v - 1}{i} \binom{k_v - m_v - 1}{i} (-\xi)^i.
$$
\label{lemRe3}
\end{lemma}

\begin{proof}
By definition of $f_v$ we have
$$
f_v \left( \bmx a&0\\0&\bar{a} \emx \gamma \bmx b&0\\0&\bar{b} \emx \right) = \Omega_v^{-1}(a b)
f_v(\gamma).
$$
Applying Lemma \ref{lemarchfn} gives the result.
\end{proof}

\subsubsection{Bounds on $I(\xi, \1_v)$} \label{bounds}

We now bound the regular orbital integral in the remaining cases. For a place $v$ of $F$, we fix $x
\in E_v$ such that $\xi=\eps_v x\bar x$.

\begin{lemma}\label{lemReBd}
Let $v$ be a finite prime.   Then
\begin{align*}
|I(\xi, \1_v)|
\leq \vol \left(F_v^\times \bs (1+\varpi_v^{n(\Omega_v)}\O_{E_v})^\times F_v^\times \right)\vol\left(I(x, E_v^1) \right).
\end{align*}
\end{lemma}
\begin{proof}
By definition,
\[
I(\xi, \1) = \int_{I(x, E^1)} \int_{F^\times \bs E^\times} \1(a \gamma_{xz})\Omega(a) d^\times a d^\times z.
\]
For each fixed $z \in I(x, E^1)$ we know that there exists an $a_{xz}\in E^\times$ such that
$a_{xz} \gamma_{xz} \in R^\times$. By  a change of variables with  $a_{xz}$,
\begin{align*}
|I(\xi, \1)| \leq \int_{I(x, E^1)} \int_{F^\times \bs E^\times} \1(a) d^\times a d^\times z.
\end{align*}
The lemma now follows by the fact that $R\cap E = \O_F+\varpi^n \O_{E}$.
\end{proof}

\begin{corollary} \label{corReBd}
Let $v$ be a finite prime not dividing $\N$ which is not split in $E$.   Then
\begin{align*}
&|I(\xi, \1_v)|
\\
\leq &\vol \left(F_v^\times \bs (1+\varpi_v^{n(\Omega_v)}\O_{E_v})^\times F_v^\times
\right)\vol\left(\left\{ z\in E_v^1: v_E(1+xz)\geq \frac{v_E(1-\xi)}{2}\right\} \right).
\end{align*}
\end{corollary}
\begin{proof}
If $a \gamma_{xz} \in F^\times R^\times$ then  there must be a
$\lambda \in F^\times$ such that
 $v(\lambda^2a\overline a)=-v(1-\xi)$ and
 $\lambda a(1+zx) \in \O_F + \varpi^n \O_E$.
Thus
\[
I(x,E^1)\subseteq \left\{ z\in E^1: v_E(1+xz)\geq \frac{v_E(1-\xi)}{2}\right\} .
\]
\end{proof}

\begin{lemma}
Let $v$ be a finite prime which  is split in $E$.

If $v(1-\xi)<0$, then
\begin{align*}
|I(\xi, \1_v)| \leq (|v(\xi)|+1)\vol \left(F_v^\times \bs (1+\varpi_v^{n(\Omega_v)}\O_{E_v})^\times
F_v^\times \right)\vol\left(U_{F_v}\right).
\end{align*}
If $v(\mathfrak{c}(\Omega))\geq v(1-\xi)\geq 0$, then
\begin{align*}
|I(\xi, \1_v)| \leq \left\{
  \begin{array}{ll}
    \vol \left(F_v^\times \bs (1+\varpi_v^{n(\Omega_v)}\O_{E_v})^\times F_v^\times
\right)\vol\left(E^1\cap x^{-1}(1+ \varpi^{\frac{v(1-\xi)}{2}}\O_{E_v})\right), & \hbox{if $v(1-\xi)$ is even;} \\
    0, & \hbox{if $v(1-\xi)$ is odd.}
  \end{array}
\right.
\end{align*}
\label{lemRe4}
\end{lemma}

\begin{proof}
Let $a=(a_1, a_2), x=(x_1, x_2), z=(b, b^{-1})$ and $k=v(1-\xi)$. We take $\tau = (1,0)$. We look
at our orbital integral,
\[
I(\xi, \1)=\int_{F^\times\backslash E^\times} \int_{E^1} \1 \left( \begin{pmatrix} a & a z x \\
\overline a \overline z \overline x & \overline a \end{pmatrix} \right) \Omega(a) d^\times a
d^\times z.
\]
It is easy to see that $a \gamma_{xz} \in R^\times$ if and only if the following conditions are satisfied:
\begin{enumerate}
\item $a_2 \in \frac{1}{\varpi^n}\O_F$
\item $a_1 \in (-a_2 + \O_F )\cap \varpi^{-k-v(a_2)}U_F$
\item $a_2(1+x_2b^{-1})\in \O_F$
\item $a_1(1+x_1b) \in a_2(1+x_2 b^{-1}) +\varpi^n \O_F$.
\end{enumerate}
For the intersection in condition (2) to be nonempty, we must have either
$v(a_2)=v(a_1)=\frac{-k}{2}$ or $-k \geq v(a_2) \geq 0$ and $v(a_1)=-k-v(a_2)$.

Thus if $k \geq 0$, then we must have $v(a_2)=v(a_1)=\frac{-k}{2}$ and hence,
\begin{align*}
I(x, E^1)\subseteq E^1\cap \frac{1}{x}(1+ \varpi^{\frac{v(1-\xi)}{2}}\O_{E_v}).
\end{align*}

If $k< 0$, then $-k \geq v(a_2) \geq 0$ and $v(a_1)=-k-v(a_2)$. In this case we must have $v(a_2
x_2 b^{-1}) \geq 0$ and $v(a_1 x_1 b) \geq 0$. Thus $v(a_2 x_2) \geq v(b)\geq -v(a_1 x_1)$.  Thus
\begin{align*}
\vol\left(I(x, E^1)\right) \leq \sum_{l=0}^{-k}\sum_{m=k+l-v(x_1)}^{l+v(x_2)}\vol(U_F) = (|v(\xi)|+1)\vol(U_F).
\end{align*}
For the last equality we are using the fact that if $k<0$, then $k=v(x_1x_2)$.

We complete the proof  by applying Lemma \ref{lemReBd}.
\end{proof}

We now bound the volume term that appears in the results above.

\begin{lemma}\label{lemvolume}
Assume that $v$ is unramified in $E$. Then,
$$
\vol(E_v^1\cap x^{-1}(1+\varpi_v^k\O_{E_v})) \leq \left\{
  \begin{array}{ll}
    \vol(E_v^1\cap U_{E_v})(1 + |v(x\bar{x})|), & \hbox{if $k = 0$ and $v$  split;} \\
    \vol(E_v^1), & \hbox{if $k = 0$ and $v$ not split;} \\
    \vol(E_v^1) q_v^{-k}(1 + q_v^{-1})^{-1}, & \hbox{if $k > 0$ and $v$ inert;} \\
    \vol(E_v^1\cap U_{E_v}) q_v^{-k}(1 - q_v^{-1})^{-1}, & \hbox{if $k > 0$ and $v$ split.}
  \end{array}
\right.
$$
If $v$ is ramified in $E$ there exists a constant $C(E_v, F_v)$ such that,
$$
\vol(E_v^1\cap x^{-1}(1+\varpi_v^k\O_{E_v})) \leq C(E_v, F_v) q_v^{-k}
$$
for all $k\geq 0$.
\end{lemma}

\begin{proof}
First we assume that $k = 0$. We note that when $E/F$ is not split we have
$$
\vol(E^1\cap x^{-1} \O_E) =
\left\{
  \begin{array}{ll}
    \vol(E^1), & \hbox{if $x\in\O_E$;} \\
    0, & \hbox{otherwise.}
  \end{array}
\right.
$$
Next we assume $E/F$ is split. We write $E = F\oplus F$ and $x = (x_1, x_2)$. Then we have
$$
x^{-1} \O_E = \left\{ (\alpha_1, \alpha_2) : v(\alpha_1) \geq -v(x_1), v(\alpha_2) \geq -v(x_2)
\right\}.
$$
Let $(y_1, y_2) \in E^1$. Then we have $v(y_1y_2) = 0$ and hence $y \in \frac 1x \O_E$ implies
that
$$
- v(x_1) \leq v(y_1) \leq  v(x_2).
$$
Hence the intersection is empty if $v(x\bar{x}) > 0$, and if $v(x\bar{x}) \leq 0$ then
$$
\vol(E^1\cap x^{-1} \O_E) = (1 + |v(x\bar{x})|) \vol(E^1\cap U_E).
$$

We now assume $k > 0$. If we assume that $E^1\cap x^{-1}(1+\varpi^k\O_E) \neq \emptyset$
then we clearly have
$$
\vol(E^1\cap x^{-1}(1+\varpi^k\O_E)) = \vol(E^1\cap (1+\varpi^k\O_E)).
$$
So we need to compute the order of
$$
(E^1\cap U_E) / E^1\cap (1+\varpi^k\O_E) \iso (E^1\cap U_E) (1 + \varpi^k\O_E) / (1 + \varpi^k\O_E).
$$
Now we have
$$
(E^1\cap U_E) (1 + \varpi^k\O_E) = \left\{ x \in U_E : N(x) \in N (1 + \varpi^k\O_E) \right\}.
$$
Hence we are looking to compute the size of the kernel of the map
$$
N : U_E/(1 + \varpi^k\O_E) \to U_F / N (1 + \varpi^k\O_E).
$$
When $E/F$ is unramified quadratic this map is surjective and we have
$$
\# U_E/(1 + \varpi^k\O_E) = q^{2k}(1 - q^{-2})
$$
and
$$
\# U_F / N (1 + \varpi^k\O_E)=\# U_F/(1 + \varpi^k\O_F) = q^{k}(1 - q^{-1})
$$
and hence the kernel has order
$
q^{k}(1 + q^{-1}).
$
Next we assume that $E = F\oplus F$, then the norm map is surjective and its kernel has order
$q^k(1 - q^{-1})$.

Finally we assume that $E/F$ is ramified. We write the discriminant as $\d_{E/F} = \p^{t+1}$. From \cite[Corollaire 1, pg 93]{serre:1962} we deduce that when $t$ is even and $k > \frac t2$,
$$
N(1 + \varpi^k\O_E) = 1 + \varpi^{k + \frac t2}\O_F,
$$
and when $t$ is odd $k\geq \frac t2$,
$$
N(1 + \varpi^k\O_E) = 1 + \varpi^{k + \frac{t+1}{2}}\O_F.
$$
Since $\# U_E/(1 + \varpi^k\O_E) = q^{2k}(1 - q^{-1})$ we see that when $k > \frac t2$, the order of the kernel of $N : U_E/(1 + \varpi^k\O_E) \to U_F / N (1 + \varpi^k\O_E)$ is at least $q^{k - \frac{t+1}{2}}$. Thus there exists some constant $C'(E)$ such that for all $k\geq 0$ the order of this kernel is at least $C'(E) q^{k}$.
\end{proof}

The following lemma is a direct consequence of Corollary \ref{corReBd} and Lemmas  \ref{lemRe4} and \ref{lemvolume}, and the fact that
$$
\vol(F_v^\times\bs F_v^\times(1+\varpi_v^{n(\Omega_v)}\O_{E_v})^\times) = \vol(U_{F_v}\bs U_{E_v}) q_v^{-n(\Omega_v)} L(1, \eta_v).
$$

\begin{lemma} \label{lemma:finalbound}
Assume $n(\Omega_v) > 0$, and let $k = \frac{v_F(1-\xi)}{2}$. There exists a constant $C(E_v, F_v)$ that is equal to one for all $v$ unramified in $E$ and such that
\begin{align*}
|I(\xi, \1_v)|
\leq & q_v^{-n(\Omega_v)} L(1, \eta_v) \vol(U_{F_v}\bs U_{E_v}) \vol(E_v^1\cap U_{E_v}) C(E_v, F_v)
\\
\times &
\begin{cases}
q_v^{-k}L(1, \eta_v) & \text{ when $k>0$}\\
1 & \text{ when $k\leq 0$ and $v$ is not split}\\
1 + |v_F(\xi)| & \text{ when $k\leq 0$ and $v$ is split.}
\end{cases}
\end{align*}
\end{lemma}

\subsection{Global calculations}
\label{globalcalculations}

We recall that
$$
I(f)=\sum_{\xi \in \epsilon NE^\times}I(\xi,f) + \vol(\A_F^\times E^\times \backslash
\A_E^\times)\left[I\left(0, f \right)+\delta (\Omega^2)I\left( \infty,f \right)\right].
$$
We now apply the local calculations of the previous sections to this global distribution. We denote
by
$$
I_{reg}(f) = \sum_{\xi \in \epsilon NE^\times}I(\xi,f),
$$
and
$$
I_{irreg}(f) = \vol(\A_F^\times E^\times \backslash \A_E^\times)\left[I\left(0, f \right)+\delta
(\Omega^2)I\left( \infty,f \right)\right].
$$
Thus
\begin{align}\label{equation:global}
I(f)=I_{reg}(f)+I_{irreg}(f).
\end{align}
We begin by computing $I_{irreg}(f)$.

\begin{proposition}
For $f$ as defined in Section \ref{testfunction},
$$
I_{irreg}(f) = \frac{2^{[F:\Q]+1} L(1, \eta) L_{S(\Omega)}(1, \eta)
\sqrt{|\Delta_F|}}{\sqrt{c(\Omega)} \sqrt{|\Delta_E|}}
\left(1+\delta (\Omega^2) \delta(\N)\prod_{v \in \Sigma_\infty}(-1)^{k_v-1}\right) \I\left(f_{\p} \right) ,
$$
where $\delta(\N) = 0$ if $\N \neq \O_F$ and $\delta(\O_F) = 1$. \label{propirreg}
\end{proposition}

\begin{proof}
By Lemmas \ref{lemN0}, \ref{lemN}, \ref{lemIrUn}, \ref{lemIrRa}, \ref{lemIrSp} and   \ref{lemI0},
and Corollary \ref{corIrIn},
$$
\vol(\A_F^\times E^\times \backslash \A_E^\times)\left[I\left(0, f \right)+\delta (\Omega^2)I\left(
\infty,f \right)\right]
$$
is equal to
$$
2 L(1,\eta)  \I\left(f_{\p} \right) \prod_{v< \infty} \vol (F_v^\times\backslash F_v^\times
(1+\varpi_v^{n(\Omega_v)}\O_{E_v})^\times)  \prod_{v \in
\Sigma_\infty}\vol(F_v^\times\backslash E_v^\times)  \left(1+\delta (\Omega^2) \delta(\N)
(-1)^{k_v-1}\right).
$$
Recall that for $v\in\Sigma_\infty$,
$$
\vol(F_v^\times\backslash E_v^\times) = \vol(\R^\times\bs\C^\times) = 2,
$$
and for $v$ finite we have
$$
\vol (F_v^\times\backslash F_v^\times (1+\varpi_v^{n(\Omega_v)}\O_{E_v})^\times) = \left\{
  \begin{array}{ll}
    \vol(U_{F_v}\bs U_{E_v}), & \hbox{if $v\not\in S(\Omega)$;} \\
    \vol(U_{F_v}\bs U_{E_v}) q_v^{-n(\Omega_v)} L(1, \eta_v), & \hbox{if $v\in S(\Omega)$.}
  \end{array}
\right.
$$
The proposition now follows.
\end{proof}

We now consider the regular terms in the trace formula. For each $v\in\Sigma_\infty$ we let
$\iota_v$ denote the corresponding embedding $F\hookrightarrow \R$. We note that $\xi \in \eps
NE^\times$ if and only if
\begin{enumerate}
  \item $\iota_v(\xi) < 0$ for all $v\in\Sigma_\infty$,
  \item $v(\xi)$ is odd for all $v \mid \N$, and
  \item $\eta_v(\xi) = 1$ for all finite $v \nmid \N$.
\end{enumerate}

We define
$$
X_\p = \left\{ \gamma \in R_\p : \varpi_\p^{-1} \gamma \not\in R \right\}.
$$
For $f_\p\in\H(G(f_\p), R^\times_\p)$ we define
$$
n(f_\p) = \max \left\{ v_\p(\det \gamma) : \gamma\in X_\p, f_\p(\gamma) \neq 0 \right\}
$$
and set $\mathfrak I(f_\p) = \p^{n(f_\p)}$.

We denote by $S(\Omega, \N, f_\p)$ the set of $\xi\in\eps NE^\times$ such that
\begin{enumerate}
  \item $v(\xi) \geq 1$ for all $v\mid\N$, and
  \item $(1 - \xi)^{-1} \in (\mathfrak c(\Omega) \d_{E/F} \mathfrak I(f_\p))^{-1}$.
\end{enumerate}

\begin{lemma}
We have
$$
I_{reg}(f) = \sum_{\xi\in S(\Omega, \N, f_\p)} I(\xi, f).
$$
Furthermore, the set $S(\Omega, \N, f_\p)$ is finite, and empty when $|\N| \geq d_{E/F}
(c(\Omega)|  \mathfrak I(f_\p)|)^{h_F}$.
\label{lemglobalreg}
\end{lemma}

\begin{proof}
The fact that $I_{reg}(f)$ is supported on $S(\Omega, \N, f_\p)$ follows from Lemmas \ref{lemReN}
and \ref{lemRe}. Suppose now that $\xi\in S(\Omega, \N, f_\p)$. We fix $0 \neq x\in \mathfrak
c(\Omega) \d_{E/F} \mathfrak I(f_\p)$, then
$$
(1 - \xi)^{-1} x = m\in \O_F
$$
and hence
$$
\xi = \frac{m-x}{m}.
$$
We now take $v\in\Sigma_\infty$ and consider $\iota_v : F \hookrightarrow \R$, then since
$\iota_v(\xi) < 0$ it follows that
$$
|\iota_v(m-x)| < |\iota_v(x)|.
$$
The finiteness of $S(\Omega, \N, f_\p)$ now follows from the finiteness of
$$
\left\{ y\in \O_F : |\iota_v(y)| < |\iota_v(x)| \text{ for all } v\in\Sigma_\infty \text{ and }
x\in \mathfrak c(\Omega) \d_{E/F} \mathfrak I(f_\p) \right\}.
$$
We note that since $v(\xi) \geq 1$ for all $v \mid \N$ we also require $m - x \in \N$. Hence
$S(\Omega, \N, f_\p)$ is empty whenever
$$
\left\{ y\in \O_F : |\iota_v(y)| < |\iota_v(x)| \text{ for all } v\in\Sigma_\infty \text{ and }
x\in \mathfrak c(\Omega) \d_{E/F} \mathfrak I(f_\p) \right\} \cap \N = \{ 0\},
$$
which is clearly the case when $|\N| \geq d_{E/F}|\mathfrak c(\Omega)  \mathfrak I(f_\p)|^{h_F}$.
\end{proof}

\begin{corollary}
For $\N$ sufficiently large, e.g. for $\N$ with absolute norm at least $d_{E/F}
\left(c(\Omega)|\mathfrak I(f_\p)|\right)^{h_F}$, we have
\begin{align*}
I(f) = \frac{2^{[F:\Q]+1} L(1, \eta) L_{S(\Omega)}(1, \eta) \sqrt{|\Delta_F|}}{\sqrt{c(\Omega)}
\sqrt{|\Delta_E|}}  \left(1+\delta (\Omega^2)
\delta(\N)\prod_{v \in \Sigma_\infty}(-1)^{k_v-1}\right) \I\left(f_{\p} \right).
\end{align*}
\label{corgeomNlarge}
\end{corollary}

We note that under certain simplifying assumptions we can explicitly compute the regular terms in
the geometric side of the trace formula.

We define, for   integers $k$ and $m$ with $|m|<k$, polynomials
$$
P_{k,m}(x) = \frac{1}{(1 - x)^{k - 1}} \sum_{i=0}^{k - |m|-1} \binom{k+m - 1}{i}\binom{k-m-1}{i} (-x)^i.
$$
For an ideal $\mathfrak a \subset \O_F$ we define
$$
R_E(\mathfrak a) = \left\{ \mathfrak b \subset \O_E : N_{E/F}(\mathfrak b) = \mathfrak a \right\}.
$$
For non-zero ideals $\mathfrak a, \mathfrak b \in \O_F$ we define
$$
\sigma(\mathfrak a ,\mathfrak b) = \# \left\{ \mathfrak c \subset \O_F : \mathfrak c | \mathfrak a
+ \mathfrak b\right\}.
$$

When $\Omega$ is unramified we regard its finite part as a character on the group
of fractional ideals of $F$.

\begin{proposition}
Assume that $\Omega$ is unramified, $f_\p=\1_\p$ and $E/F$
is unramified at the even places of $F$. Let $d\in\O_F$ be a generator of $\d_{E/F}$. Then we have
$I(f)$ equal to the sum of
$$
\frac{2^{[F:\Q]+1} L(1, \eta) \sqrt{|\Delta_F|}}{\sqrt{c(\Omega)} \sqrt{|\Delta_E|}} \left(1+\delta
(\Omega^2) \delta(\N)\prod_{v \in \Sigma_\infty}(-1)^{k_v-1}\right)
$$
and
$$
4^{[F:\Q]} \frac{|\Delta_F|}{|\Delta_E|} \sum_{n} |R_E(n\N^{-1})| \sigma(\d_{E/F}, (n+d))
\sum_{\mathfrak a\in R_E((n+d))} \Omega(\mathfrak D_{E/F}^{-1}\mathfrak a)  \prod_{v\in\Sigma_\infty}
P_{k_v, m_v}\left(\iota_v\left(\frac{n}{n+d}\right)\right),
$$
with the outer sum taken over the finite set of $n\in \N$ such that
\begin{enumerate}
  \item $\eta_v\left(1 + \frac dn\right) = 1$ for all $v \mid \d_{E/F}$, and
  \item $\iota_v(n)$ lies between $-\iota_v(d)$ and $0$ for all $v\in\Sigma_\infty$.
\end{enumerate}
\label{propglobalunram}
\end{proposition}

\begin{proof}
Let $\xi\in\eps NE^\times$ and let $v$ be a finite place of $F$. When $v$ splits in $E$ we have,
with the notation of Lemma \ref{lemRe2},
\begin{itemize}
  \item $v(1 - \xi) > 0 \Longrightarrow I(\xi, f_v) = 0$,
  \item $v(1 - \xi) = 0 \Longrightarrow I(\xi, f_v) = (1 + v(\xi)) \vol(U_{F_v})^2$,
  \item $v(1 - \xi) < 0 \Longrightarrow I(\xi, f_v) =  \sum_{i=0}^{|v(\xi)|} \Omega(\varpi_v^i,\varpi_v^{|v(\xi)|-i}) \vol(U_{F_v})^2$.
\end{itemize}
When $v \nmid \N$ is inert and unramified in $E$ we have $\Omega_v$ trivial and, by Lemma \ref{lemRe1},
\begin{itemize}
  \item $v(1 - \xi) > 0 \Longrightarrow I(\xi, f_v) = 0$,
  \item $v(1 - \xi) \text{ odd and } \leq 0 \Longrightarrow I(\xi, f_v) = 0$,
  \item $v(1 - \xi) \text{ even and } \leq 0 \Longrightarrow I(\xi, f_v) = \vol(U_{F_v}\bs U_{E_v})^2$.
\end{itemize}
When $v \mid \N$ we have $\Omega_v$ trivial and, by Lemma \ref{lemReN},
\begin{itemize}
  \item $v(\xi) \leq 0 \Longrightarrow I(\xi, f_v) = 0$,
  \item $v(\xi) \geq 1 \Longrightarrow I(\xi, f_v) = \vol(U_{F_v}\bs U_{E_v})^2$.
\end{itemize}
When $v$ is odd and ramified in $E$ we have $\Omega_v^2$ trivial and, by Lemma \ref{lemRe1},
\begin{itemize}
  \item $v(1 - \xi) > 1 \Longrightarrow I(\xi, f_v) = 0$,
  \item $v(1 - \xi) = 1 \Longrightarrow I(\xi, f_v) = \vol(U_{F_v}\bs U_{E_v})^2 \Omega(\varpi_{E_v})^{-1}$,
  \item $v(1 - \xi) \leq 0 \Longrightarrow I(\xi, f_v) = 2 \Omega(\varpi_{E_v})^{\frac{v_E(1-\xi)}{2}}\vol(U_{F_v}\bs U_{E_v})^2$.
\end{itemize}
And at $v\in\Sigma_\infty$ we have, by Lemma \ref{lemRe3},
$$
I(\xi, f_v) = \vol(F_v^\times\bs E^\times_v)^2 P_{m_v, k_v}(\iota_v(\xi)).
$$

As in the proof of Lemma \ref{lemglobalreg} we note that if $I(\xi, f) \neq 0$ then we have
$$
\xi = \frac{n}{n + d}
$$
with $n\in\N$. Furthermore such an element lies in $\eps NE^\times$ if and only if
\begin{itemize}
  \item $R_E(n\N^{-1})$ and $R_E((n+d))$ are non-empty,
  \item $\eta_v\left(1 + \frac dn\right) = 1$ for all $v \mid \d_{E/F}$, and
  \item $\iota_v(n)$ lies between $\iota_v(0)$ and $-\iota_v(d)$ for all $v\in\Sigma_\infty$.
\end{itemize}
Suppose now we fix such an $n$. We have
$$
I(\xi, f) = \prod_v I(\xi, f_v),
$$
and we now determine the contribution to this product from each place $v$ depending on its behavior
in the extension $E$.

For a finite place $v$ of $F$ and an ideal $\a \subset \O_F$ we define
$$
R_{E_v}(\a) = \left\{ \mathfrak b \subset \O_{E_v} : N_{E_v/F_v}\mathfrak b = \a \O_{F_v} \right\}.
$$
The contribution from the places $v \mid \N$ is,
$$
\prod_{v | \N} \vol(U_{F_v}\bs U_{E_v})^2 |R_{E_v}(n\N^{-1})|,
$$
and furthermore for such $v$ we have $R_{E_v}((n+d)) = \{\O_{E_v}\}$. For finite $v \nmid \N$ which
are inert in $E$ we get
$$
\prod_{v \nmid \N, inert} \vol(U_{F_v}\bs U_{E_v})^2 |R_{E_v}((n))| |R_{E_v}((n+d))|.
$$
For finite $v$ which split in $E$ we get
$$
\prod_{v < \infty, split} \vol(U_{F_v}\bs U_{E_v})^2 |R_{E_v}((n))| \sum_{\a\in R_{E_v}((n+d))}
\Omega_v(\a),
$$
since for such $v$ when $v(1-\xi) = 0$ we have $v(\xi) = v(n)$ and $v(n+d) = 0$, and when $v(1-\xi)
< 0$ we have $v(n) = 0$ and $|v(\xi)| = v(n+d)$. For $v$ ramified in $E$ the contribution to
$I(\xi, f)$ is
$$
\sigma(\d_{E/F}, (n+d)) \prod_{v | \d_{E/F}} \vol(U_{F_v}\bs U_{E_v})^2 |R_{E_v}((n))| \sum_{\a\in
R_{E_v}((n+d))} \Omega_v(\mathfrak D_{E/F}^{-1} \a),
$$
since for such $v$, $\Omega_v^2$ is trivial and $v(1-\xi) = v(\d_{E/F}) - v(n+d)$. Finally for
$v\in\Sigma_\infty$ we have $I(\xi, f_v) = 4 P_{k_v, m_v}(\iota_v(n/(n+d)))$.

Putting these local calculations together gives the sum in the statement of the Proposition and
combining it with the calculation of the irregular terms from Proposition \ref{propirreg} gives the
result.
\end{proof}

\section{A measure on the Hecke algebra}

The goal of this section is to find the measure $\mu$ such that  for $f_\p \in \mathcal H(G_\p,
Z_\p R_\p^\times)$,
\[
\I(f_\p)=\int_{-2}^{2}\hat f_\p(x)\mu,
\]
where $\I(f_\p)$ is defined in Section \ref{irregularcosets}. We begin by recalling standard facts
about distributions on Hecke algebras which come from the Plancherel formula; alternatively see
\cite[Section 2]{serre:1997}.

\subsection{An application of the Plancherel formula}

We assume throughout this subsection that $F$ is a non-archimedean local field of characteristic
zero. We let $q$ denote the order of the residue field of $F$. We let $G' = \PGL(2, F)$ and $K =
\PGL(2, \O_{F})$. We fix a Haar measure on $G'$ which gives $K$ volume one. For $n\geq 0$ we denote
by $f_n$ the characteristic function of
$$
K \bmx \varpi^n&0\\0&1 \emx K.
$$

For $s\in i\R$ let $\pi_s$ denote the unramified principal series representation of $G'$ unitarily
induced from
$$
\bmx a&0\\0&b \emx \mapsto \left| \frac{a}{b} \right|^s.
$$
For $f\in\H(G', K)$ and $x \in [-2, +2]$ we define
$$
\hat{f}(x) = \Tr \pi_s(f),
$$
where $x = q^{s} + q^{-s}$. We have (see \cite[Lemma 9]{ramakrishnan:2005}) $\hat{f}_0 \equiv 1$
and for $n > 0$,
\begin{equation}
\hat{f}_n(x) = q^{\frac n2} \left( q^{ns} + q^{-ns} + (1 - q^{-1})(q^{(n-2)s} + q^{(n-4)s} + \hdots
+ q^{-(n-2)s})\right). \label{fnhat}
\end{equation}

On the interval $[-2, +2]$ we take the Sato-Tate measure
$$
\mu_\infty = \frac{\sqrt{4-x^2}}{2\pi} \ dx
$$
and the spherical Plancherel measure on $\PGL(2, F)$,
$$
\mu_q = \frac{q + 1}{(q^{\frac 12} + q^{-\frac 12})^2 - x^2} \mu_\infty.
$$

By the Plancherel formula we have,

\begin{lemma}
For all $f\in \H(G', K)$ we have
$$
f \bmx 1&0\\0&1 \emx = \int_{-2}^{2} \hat{f}(x) \ \mu_q,
$$
and
$$
f \bmx \varpi^m&0\\0&1 \emx = \frac{1}{(1 + q^{-1}) q^{m}} \int_{-2}^{2} \hat{f}(x) \hat{f}_m(x) \
\mu_q,
$$
for $m > 0$.
\label{lemmeasures}
\end{lemma}

In particular we note the following corollary.

\begin{corollary}
We have
$$
f \bmx 1&0\\0&1 \emx + \delta f \bmx \varpi&0\\0&1 \emx = \frac{1}{q + 1} \int_{-2}^{2} \hat{f}(x)
(1 + q^{\frac 12}\delta x + q) \ \mu_q,
$$
and
$$
f \bmx 1&0\\0&1 \emx - f \bmx \varpi^2&0\\0&1 \emx = \int_{-2}^{2} \hat{f}(x) \ \mu_\infty.
$$
\label{cormeasures}
\end{corollary}

We define for $\delta\in\C$ a distribution
$$
\Lambda_\delta : f \mapsto \sum_{n\in\Z} \delta^n f \bmx \varpi^n&0\\0&1 \emx
$$
on $\H(G', K)$. Then we have the following.

\begin{lemma}
For $f\in\H(G', K)$ and $\delta$ with $|\delta| = 1$,
$$
\Lambda_\delta(f) = \int_{-2}^{2} \hat{f}(x) \frac{(1 - q^{-1})}{(1 - \delta q^{-\frac 12} x +
\delta^2 q^{- 1}) (1 - \delta^{-1} q^{-\frac 12} x + \delta^{-2} q^{- 1})} \ \mu_\infty.
$$
\label{lemmeasures1}
\end{lemma}

\begin{proof}
This follows by a short calculation using Lemma \ref{lemmeasures}, Formula \ref{fnhat} and that,
$$
f \bmx \varpi^m&\\&1 \emx = f \bmx \varpi^{-m}&\\&1 \emx.
$$
Alternatively, one can argue as in \cite[Section 6]{ramakrishnan:2005} in the case $\delta = 1$.
\end{proof}

\subsection{The distribution $\tilde{I}$}
\label{distributionI}

For $\alpha, \beta \in \C$ we let $\rho(\alpha, \beta)$ denote the unramified representation of
$\GL(2, F_\p)$ with Satake parameters $\{ \alpha, \beta\}$. For $x\in [-2, +2]$ we define
$\alpha_x\in\C$ to be such that $\alpha_x + \alpha_x^{-1} = x$; of course $\alpha_x$ is not well defined, however, all constructions below will depend only on $x$, and not on the choice of $\alpha_x$. We then define
$$
\mu_{\p, E, \Omega} = L(1/2, \rho(\alpha_x,\alpha_x^{-1})_{E_\p}\otimes\Omega_\p)  \frac{\sqrt{4 -
x^2}}{2\pi} \ dx,
$$
where $\rho(\alpha_x,\alpha_x^{-1})_{E_\p}$ denotes the base change of
$\rho(\alpha_x,\alpha_x^{-1})$ to $\GL(2, E_\p)$. We note that,
\begin{align*}
& L(1/2, \rho(\alpha_x,\alpha_x^{-1})_{E_\p}\otimes\Omega_\p) = 1, \text{ when $\Omega_\p$ is
ramified and},\\
& L(1/2, \rho(\alpha_x,\alpha_x^{-1})_{E_\p}\otimes\Omega_\p) =\\
& \left\{
  \begin{array}{ll}
    (1 - x\Omega(\varpi_{E_\p})q_\p^{-\frac 12} + \Omega(\varpi_{E_\p})^2 q_\p^{-1})^{-1} (1 - \frac{x}{\Omega(\varpi_{E_\p})} q_\p^{-\frac 12} + \frac{1}{\Omega(\varpi_{E_\p})^2} q_\p^{-1})^{-1}, & \hbox{for $\p$ split,} \\
    ((1 + q_\p^{-1})^2 - x^2 q_\p^{-1})^{-1}, & \hbox{for $\p$ inert,} \\
    (1 - x \Omega(\varpi_{E_\p}) q_\p^{-\frac 12} + q_\p^{-1})^{-1}, & \hbox{for $\p$ ramified,}
  \end{array}
\right.
\end{align*}
when $\Omega_\p$ is unramified, where $\varpi_{E_\p}$ is defined in Section \ref{geometric}. We
note that
$$
\int_{-2}^2 \mu_{\p, E, \Omega} =
\left\{
  \begin{array}{ll}
    L(1, \eta_\p), & \hbox{if $\Omega$ is unramified at $\p$;} \\
    1, & \hbox{otherwise.}
  \end{array}
\right.
$$
We note for future reference the relationship between $\mu_{\p, E, \Omega}$ and the spherical
Plancherel measure $\mu_\p$ on $\PGL(2, F_\p)$ is given by
$$
\mu_{\p, E, \Omega} = \frac{L(1/2, \rho(\alpha_x, \alpha_x^{-1})_{E_\p}\otimes\Omega_\p)}{L(1,
\rho(\alpha_x, \alpha_x^{-1}), Ad)} L(2, 1_{F_\p}) \mu_\p.
$$

We recall $\I(f_\p)$ is defined in Section \ref{irregularcosets}.

\begin{lemma}
For all $f_\p \in \H(G(F_\p), Z_\p R_\p^\times)$ we have
$$
\I(f_\p) = \frac{\int_{-2}^2 \hat{f}_\p(x) \ \mu_{\p, E, \Omega}}{L(1, \eta_\p)},
$$
if $\Omega$ is unramified at $\p$ and we have
$$
\I(f_\p) = \int_{-2}^2 \hat{f}_\p(x) \ \mu_{\p, E, \Omega},
$$
if $\Omega$ is ramified at $\p$.
\label{lemItilde}
\end{lemma}

\begin{proof}
We fix an isomorphism $D_\p \iso M(2, F_\p)$ and an embedding $E_\p \hookrightarrow M(2, F_\p)$
such that $M(2, \O_{F_\p}) \cap E = \O_F + \varpi_\p^{n(\Omega_\p)} \O_{E_\p}$ as in the proof of
Lemma \ref{lemdoublecoset}. We now prove the lemma on a case by case basis.

First we assume that $n(\Omega_\p) > 0$. Then we have, by Lemmas \ref{lemIrUn}, \ref{lemIrRa}, \ref{lemIrSp} and Corollary
\ref{cormeasures},
$$
\tilde{I}(f_\p) = f_\p \bmx 1&0\\0&1 \emx - f_\p \bmx \varpi_\p^2&0\\0&1 \emx = \int_{-2}^{2}
\hat{f}_\p(x) \mu_\infty.
$$
On the other hand in this case we clearly have $\mu_{\p, E, \Omega} = \mu_\infty$.

Next we assume that $\p$ is unramified and inert in $E$ and $\Omega$ is unramified at $\p$. Then
from Lemmas \ref{lemIrUn} and \ref{lemmeasures} we have
\begin{align*}
\tilde{I}(f_\p) = \int_{-2}^{2} \hat{f}_\p(x) \mu_{q_\p} = \int_{-2}^{2} \hat{f}_\p(x) \mu_{\p, E,
\Omega}.
\end{align*}

Next we assume that $\p$ splits in $E$ and $\Omega$ is unramified above $\p$. We write $\Omega_\p =
(\chi, \chi^{-1})$. Then from Lemmas \ref{lemIrSp} and \ref{lemmeasures1} we have
\begin{align*}
\tilde{I}(f_\p) = \sum_{m\in\Z} \chi(\varpi_\p^m) f_\p \bmx \varpi_\p^m&0\\0&1 \emx = \int_{-2}^2
\hat{f}_\p(x) \mu_{\p, E, \Omega}.
\end{align*}

Finally we assume that $\p$ is ramified in $E$ and $\Omega$ is unramified at $\p$. Then
from Lemma \ref{lemIrRa} and Corollary \ref{cormeasures} we have
\begin{align*}
\tilde{I}(f_\p) = f \bmx 1&0\\0&1 \emx + \Omega(\tau) f\bmx \varpi_\p&0\\0&1 \emx = \int_{-2}^{2}
\hat{f}_\p(x) \mu_{\p, E, \Omega}.
\end{align*}
\end{proof}

\section{Main results}
\label{mainresults}

We now combine the calculations of Sections \ref{spectral} and \ref{geometric} to obtain the main results of this paper.

\subsection{Average $L$-values}

By Propositions \ref{propspectral}, \ref{propirreg} and Lemma \ref{lemglobalreg} we see that we have an exact formula for
\[
\sum_{\pi\in\F(\N, 2\k)}
\frac{L(1/2, \pi_E\otimes\Omega)}{L(1, \pi, Ad)} \hat{f}_\p(\pi_\p),
\]
in terms of orbital integrals. We now write down the formula precisely under certain further
assumptions for which we have computed all the necessary orbital integrals.

Combining Proposition \ref{propspectral}, Corollary \ref{corgeomNlarge} and Lemma \ref{lemItilde} we get the following.

\begin{theorem} \label{thmlargelevel}
Let $E$ be a CM extension of a totally real number field $F$. Let $\N \subset \O_F$ be an ideal such that each prime dividing $\N$ is unramified and inert in $E$ and that the number of primes dividing $\N$ has the same parity as $[F:\Q]$. Let $\Omega : \A_F^\times E^\times \bs \A_E^\times\to \C^\times$ be a character which is unramified outside of $\N$ and such that for $v\in\Sigma_\infty$ the weight $m_v$ of $\Omega_v$ is strictly less than $k_v$. Let $f_\p\in\H(G(F_\p), Z_\p R^\times_\p)$. Then for
$|\N| \geq d_{E/F} (c(\Omega) |\mathfrak I(f_\p)|)^{h_F}$,
$$
\frac{2^{[F:\Q]}}{|\N|} \binom{2\k - 2}{\k + \m - 1} \sum_{\pi\in\F(\N, 2\k)}
\frac{L(1/2, \pi_E\otimes\Omega)}{L(1, \pi, Ad)} \hat{f}_\p(\pi_\p),
$$
is equal to
\begin{multline*}
{4 {|\Delta_F|}^{\frac{3}{2}}L^{S(\Omega)\cup\{\p\}}(1, \eta)} \left( 1 + \delta(\Omega^2) \delta(\N)
\prod_{v\in\Sigma_\infty} (-1)^{k_v-1} \right) \int_{-2}^{2} \hat{f}_\p(x)
\mu_{\p, E, \Omega}\\
- 4 C(\k, \Omega, f_\p) L^{S(\Omega)}(1, \eta)^2 \frac{\sqrt{c(\Omega) |\Delta_E|}}{L(2,1_F)
\sqrt{|\Delta_F|}} \prod_{v \mid \N}\frac{1}{q_v-1}\prod_{v\in\Sigma_\infty} \frac{2 \pi}{2k_v-1},
\end{multline*}
where $\mathfrak I(f_\p)$ is defined before Lemma 4.21, $C(\k, \Omega, f_\p)$ is defined in Lemma \ref{lemcont} and the measure $\mu_{\p, E,
\Omega}$ is defined in Section \ref{distributionI}.
\end{theorem}

We note in particular that when $f_\p$ is the identity in $\H(G(F_\p), Z_\p R^\times_\p)$ the second
term in the theorem above is equal to
$$
{4 {|\Delta_F|}^{\frac{3}{2}}L^{S(\Omega)}(1, \eta)} \left( 1 + \delta(\Omega^2) \delta(\N)
\prod_{v\in\Sigma_\infty} (-1)^{k_v-1} \right).
$$

We recall that by the Ramanujan conjecture $a_\p(\pi) \in [-2, 2]$ for all $\pi \in \F(\N, 2\k)$; see \cite{blasius:2006} for the most general version. The distribution of the $\a_\p(\pi)$ has been considered by Sarnak \cite{sarnak:1987} and Serre \cite{serre:1997}. In \cite{serre:1997} it is proven, when $F=\Q$, that as $\N \to \infty$ the set
$$
\{ a_\p(\pi) : \pi \in \F(\N, 2\k) \}
$$
becomes equidistributed with respect to the measure
$$
\mu_\p = \frac{q_\p + 1}{(q_\p^{\frac 12} + q_\p^{- \frac 12})^2 - x^2} \frac{\sqrt{4-x^2}}{2\pi}
dx
$$
on $[-2, 2]$. That is, for all $J \subset [-2, 2]$,
$$
\lim_{\N \to \infty} \frac{1}{\# \F(\N, 2\k)} \sum_{\substack{\pi\in\F(\N, 2\k) \\
a_\p(\pi)\in J}} 1 = \mu_\p(J).
$$
We note that $\mu_\p$ comes from the spherical Plancherel measure on $\PGL(2, F_\p)$.

Using Theorem \ref{thmlargelevel} we obtain a variant of this equidistribution result where we
include a weighting by $L^\p(1/2, \pi_E\otimes\Omega)$.

\begin{corollary}
Let $J \subset [-2, +2]$. Then we have
$$
\lim_{|\N| \to \infty} \frac{1}{|\N|} \sum_{\substack{\pi\in\F(\N, 2\k) \\ a_\p(\pi)\in J}}
\frac{L^\p(1/2, \pi_E\otimes\Omega)}{L^\p(1, \pi, Ad)}
$$
equal to
$$
{4}{{|\Delta_F|}^{\frac{3}{2}}} \frac{1}{2^{[F:\Q]}} \binom{2\k - 2}{\k + \m - 1}^{-1}
L^{S(\Omega)\cup\{\p\}}(1, \eta) L(2, 1_{F_\p}) \mu_{\p}(J).
$$
Here the limit is taken over squarefree ideals $\N$ prime to $\c(\Omega)$ and such that each prime dividing $\N$ is inert and unramified in $E$ and the number of primes dividing $\N$ has the same parity as $[F:\Q]$. \label{corJ}
\end{corollary}

\begin{proof}
This result follows by an  application of \cite[Proposition 2]{serre:1997} to Theorem \ref{thmlargelevel}.
\end{proof}

We note when $\Omega$ is trivial and $F=\Q$, we recover the main result of
\cite{ramakrishnan:2005}. From this result we conclude the following.

\begin{corollary}
There are holomorphic cusp forms $\pi$ on $\PGL(2)/F$ of squarefree level and fixed weight such
that $a_\p(\pi)$ lies arbitrarily close to $\pm 2$ and $L(1/2, \pi_E\otimes\Omega) \neq 0$.
\end{corollary}

It is worth remarking that one could consider the average by normalizing by  $|\F(\N, 2\k)|$ rather than $|\N|$.  In order to have a finite limit with this normalization we need to add the technical restriction that
$|\N|\prod_{\p | \N}(1-\frac{1}{|\p|})\sim |\N|$.  For $F=\Q$ this condition reduces to  $\phi(N) \sim N$ where $\phi$ is the Euler totient function.  Using the well known fact that $|\F(N, 2k)| \sim  \frac{2k-1}{12}\phi(N)$ as
$N\rightarrow \infty$ we get the following statement.

\begin{corollary}
Let $F=\Q$ and $J \subset [-2, +2]$. Then
\[
\lim_{N \to \infty} \frac{1}{|\F(N, 2k)|} \sum_{\substack{\pi\in \F(N, 2k) \\ a_p(\pi)\in J}}
\frac{L^p(1/2, \pi_E\otimes\Omega)}{L^p(1, \pi, Ad)}
\]
is equal to
\[
\frac{24}{2k-1}  \binom{2k - 2}{k + m - 1}^{-1}
L^{S(\Omega)\cup\{p\}}(1, \eta) L(2, 1_{\Q_p}) \mu_{p}(J)
\]
where the limit is taken over squarefree $N$ such that $\phi(N) \sim N$ and each prime dividing $N$ is inert and unramified in $E$ and does not divide $c(\Omega)$.
\end{corollary}
In principal, one could use the formulas developed in this paper to study the average computed by summing over the space of cusp forms and normalizing by the dimension of this space. Averaging in this way  may remove the condition that $\phi(N) \sim N$.

Finally, when $\Omega$ is unramified we have the following exact formula for all levels by
combining Propositions \ref{propspectral} and \ref{propglobalunram}.

\begin{theorem}
Let $E$ be a CM extension of a totally real number field $F$. Let $\N \subset \O_F$ be an ideal such that each prime dividing $\N$ is unramified and inert in $E$ and such that the number of primes dividing $\N$ has the same parity as $[F:\Q]$. Let $\Omega : \A_F^\times E^\times \bs \A_E^\times\to \C^\times$ be a character which is everywhere unramified and such that for $v\in\Sigma_\infty$ the weight $m_v$ of $\Omega_v$ is strictly less than $k_v$. Assume furthermore that $E/F$ is unramified at the even places of $F$. Let $d\in\O_F$ be a generator of $\d_{E/F}$. Then, with $C(\k, \Omega, \1_\p)$ defined as in Lemma \ref{lemcont},x
\begin{multline*}
\frac{1}{ \sqrt{c(\Omega) d_{E/F}}|\N||\Delta_F|^2} \binom{2\k - 2}{\k + \m - 1} \sum_{\pi\in\F(\N,
2\k)} \frac{L(1/2, \pi_E\otimes\Omega)}{L(1, \pi, Ad)}\\
+ C(\k, \Omega, \1_\p) \frac{4 L(1, \eta)^2 }{{|\Delta_F|}^{\frac{3}{2}}L(2,1_F) }
\prod_{v \mid \N}\frac{1}{q_v-1}\prod_{v\in\Sigma_\infty} \frac{ \pi}{2k_v-1},
\end{multline*}
is equal to the sum of
$$
\frac{4 L(1, \eta) \sqrt{|\Delta_F|}}{2^{[F:\Q]}\sqrt{c(\Omega)} \sqrt{|\Delta_E|}} \left(1+\delta
(\Omega^2) \delta(\N)\prod_{v \in \Sigma_\infty}(-1)^{k_v-1}\right)
$$
and
$$
 2\frac{|\Delta_F|}{|\Delta_E|} \sum_{n} \left( |R_E(n\N^{-1})| \sigma(\d_{E/F}, (n+d))
\sum_{\mathfrak a\in R_E((n+d))} \Omega(\mathfrak D_{E/F}^{-1}\mathfrak a) \times \prod_{v\in\Sigma_\infty}
P_{k_v, m_v}\left(\iota_v\left(\frac{n}{n+d}\right)\right)\right),
$$
with the outer sum taken over the finite set of $n\in \N$ such that
\begin{enumerate}
  \item $\eta_v\left(1 + \frac dn\right) = 1$ for all $v \mid \d_{E/F}$, and
  \item $\iota_v(n)$ lies between $-\iota_v(d)$ and $0$ for all $v\in\Sigma_\infty$.
\end{enumerate}
\label{thmunramified}
\end{theorem}

\subsection{Subconvexity}
\label{subconvexity}

We now apply our calculations of the relative trace formula to the problem of subconvexity.

Let $\pi_1$ and $\pi_2$ be cuspidal automorphic representations of $\GL(2, \A_F)$. The convexity bound for $L_{fin}(1/2, \pi_1\times\pi_2)$ is that for $\eps > 0$,
$$
L_{fin}(1/2, \pi_1\times\pi_2) \ll_\eps C(\pi_1\times\pi_2)^{\frac 14 + \eps},
$$
where $C(\pi_1\times\pi_2) = C_{fin}(\pi_1\times\pi_2) C_\infty(\pi_1\times\pi_2)$ is the analytic conductor of $\pi_1 \times \pi_2$; see \cite[Section 2.A]{iwaniec:2000b}. $C_{fin}(\pi_1\times\pi_2)$ is the conductor of $\pi_1\times\pi_2$ and $C_\infty(\pi_1\times\pi_2)$ depends only on the infinity types of $\pi_1$ and $\pi_2$; we refer to \cite[Section 2.A]{iwaniec:2000b} for the precise definition. We note that when $\pi_1$ and $\pi_2$ have disjoint ramification $C_{fin}(\pi_1\times\pi_2) = (C_{fin}(\pi_1)C_{fin}(\pi_2))^2$.

The problem of beating the convex bound, with $\pi_2$ fixed, has been the study of many authors. When $F=\Q$ and $\pi_1$ and $\pi_2$ have trivial central character the convexity bound was beaten by Kowalski, Michel, and VanderKam \cite{kowalski:2002}, with the central character condition being relaxed by Michel and Harcos \cite{michel:2004}, \cite{harcos:2006}. In the case that $\Omega$ is trivial, so that the $L$-function factors as $L(s, \pi_E) = L(s, \pi) L(s, \pi \otimes \eta)$, the convexity bound was beaten by Duke, Friedlander and Iwaniec \cite{duke:1994} in the level aspect over $\Q$ with $\pi$ fixed and $\eta$ varying. For number fields other than $\Q$ the first subconvex result was obtained by Cogdell, Piatetski-Shapiro and Sarnak \cite{cogdell:2003} in the case of a fixed Hilbert modular form twisted by a ray class character. Further extensions of these subconvexity results to cusp forms on arbitrary number fields have been obtained by Venkatesh \cite{venkatesh:2005}.

We now continue with the usual assumptions on $\pi \in \F(\N, 2\k)$ and $\Omega$ as in Section \ref{notation}. We denote by $\sigma_\Omega$ the induction of $\Omega$ to an automorphic representation of $\GL(2, \A_F)$. Hence,
\[
L(s, \pi_E\otimes\Omega) = L(s, \pi \times \sigma_\Omega).
\]
We have $C_{fin}(\pi) = |\N|$ and, by the formula for the conductor of an induced representation (see for example \cite[Section 1.2]{schmidt:2002}), $C_{fin}(\sigma_\Omega) = d_{E/F} c(\Omega)$. Thus, $C_{fin}(\pi\times\sigma_\Omega) = (|\N| d_{E/F} c(\Omega))^2$, and the convexity bound in the level aspect is given by
$$
L_{fin}(1/2, \pi\times\sigma_\Omega) \ll_{\k, \eps} (|\N| d_{E/F} c(\Omega))^{\frac 12 + \eps}.
$$

We now proceed to apply our work to the problem of beating convexity for these $L$-functions. By combining our calculations of the relative trace formula, together with the bounds on the orbital integral integrals from Section \ref{bounds}, we will get an estimate for $L_{fin}(1/2, \pi\times\sigma_\Omega)$ (Theorem \ref{thmsub} below) which beats the convexity bound as $\pi$ and $\Omega$ vary in a hybrid range.

The relative trace formula provides an expression for the first moment of $L(1/2, \pi\times\sigma_\Omega)$ averaged over $\pi \in \F(\N, 2\k)$. The size of this family is approximately $|\N|$ while the conductor of the $L$-function is $(|\N| d_{E/F} c(\Omega))^2$, this allows us to obtain estimates which beat convexity when $|\N|$ is of size around $\sqrt{c(\Omega)}$. We note that over $\Q$, Michel \cite{michel:2004} and Harcos-Michel \cite{harcos:2006} in their work on subconvexity for $L(s, \pi_1\times\pi_2)$ average over the family of modular forms of level $[C_{fin}(\pi_1), C_{fin}(\pi_2)]$ and bound the second moment of the $L$-function. By shortening the family we are able to get away with an estimate for only the first moment.

Recall that $\Omega$ is a unitary character on the idele class group of $E$ that is trivial when restricted to $\mathbf A_F^\times$, unramified above $\N$ and has weight $|m_v|$ strictly less than $k_v$ for every infinite place $v$ of $F$. We let $S(\Omega)=\{\p_1, \p_2,...,\p_m\}$ denote the places of $F$ above which $\Omega$ is ramified, $\c(\Omega)$ the norm of the conductor of $\Omega$ in $F$ and $c(\Omega)$ the absolute norm of $\c(\Omega)$.
Let $m=|S(\Omega)|, \c(\Omega)=\prod_{i=1}^m \p_i^{2n_i}$ and $N_{F/\Q}\p_i=q_i$.

In the remainder of this section we fix $E$ and allow $\N$ and $\Omega$ to vary. We recall that $D$ ramifies precisely at the infinite places of $F$ and the places dividing $\N$. Therefore as $\N$ varies, $D$ varies and hence the image of $E$ in $D$ depends on $\N$. However $E$ as a field extension of $F$ does not depend on $\N$. For the bounds we prove in this section, the $\ll_E$ notation refers to constants that only depend on $E$ as a field, such at $\sqrt{|\Delta_E|}$, and thus have no hidden dependence on $\N$.

We begin with a necessary technical lemma which we will require in the bounding of the geometric expansion of $I(f)$. For $a\in F$ and integers $r_i, t_i \geq 0$ for $1\leq i\leq m$ let
\[
S_{(r_i), (t_i)}(a)= \{ y \in \N: v_{\p_i}(y)=r_i, v_{\p_i}(y-a)=t_i,  |\iota_v(y)| < |\iota_v(a)|, i=1,...,m, v \in \Sigma_\infty \}.
\]

\begin{lemma} \label{claim:size}
$S_{(r_i),(t_i)}(a)$ is empty unless for each $i=1,..., m$,
\begin{enumerate}
\item $r_i < v_{\p_i}(a)$ and $t_i=r_i$ or
\item $r_i>v_{\p_i}(a)$ and $t_i = v_{\p_i}(a)$ or
\item $ r_i=v_{\p_i}(a)$ and $t_i \geq v_{\p_i}(a)$.
\end{enumerate}
In addition,
\begin{align*}
|S_{(r_i),(t_i)}(a)|  \leq
\frac{c(F)|N_{F/\Q}(a)|}{|N_{F/\Q}(\N \p_1^{\max \{ r_1, t_1\}} ... \p_m^{\max\{r_m, t_m\}})|}
\end{align*}
where $c(F)$ is a constant that only depends on $F$.
\end{lemma}
\begin{proof}
It is clear that $S_{(r_i),(t_i)}(a)$ is empty unless conditions 1, 2 or 3 hold for each $i$. We now proceed to prove the bound following the ideas of the proof in \cite[\S V.1, Theorem 0]{lang:1994}.

Since $S_{(r_i),(t_i)}(a)$ is finite there exists an integer $n$ such that
\[
n< |S_{(r_i),(t_i)}(a)| \leq n+1.
\]
Pick a $v_0 \in \Sigma_\infty$. Identify $F_{v_0}$ with the real line. By the condition $| \iota_{v_0}(y)| < |\iota_{v_0}(a)|$, $S_{(r_i),(t_i)}(a)$ is contained in the interval centered at the origin of length $2|\iota_{v_0}(a)|$. By the Pigeon Hole Principle if we divide this interval into $n$ equal subintervals there must be at least two distinct elements $x, y \in S_{(r_i),(t_i)}(a)$ such that $x$ and $y$ are in the same subinterval. Thus
\[
|\iota_{v_0}(x-y)| \leq \frac {2|\iota_{v_0}(a)|}{n}.
\]
For all $v \in \Sigma_\infty$ with $v$ not equal to $v_0$,
\[
|\iota_v(x-y)| \leq 2|\iota_v(a)|.
\]
We also know that
\[
|x-y|_{\p_i}\leq q_i^{-r_i}
\]
 and
 \[
|x-y|_{\p_i}= |(x-a)-(y-a)|_{\p_i} \leq q_i^{-t_i}.
\]
Thus
\[
|x-y|_{\p_i} \leq q_i^{-\max\{r_i, t_i\}}.
\]
For $v | \N$,
\[
|x-y|_v \leq |\N|_v.
\]
Thus
\[
1= \prod_v |x-y|_v \leq \frac{2^{[F:\Q]} |N_{F/\Q}(a)| }{n|N_{F/\Q}(\N \p_1^{\max \{ r_1, t_1\}} ... \p_m^{\max\{r_m, t_m\}})|}.
\]
Hence
\[
|S_{(r_i),(t_i)}(a)| \leq n+1 \leq 2n \leq \frac{2^{[F:\Q] + 1} |N_{F/\Q}(a)| }{|N_{F/\Q}(\N \p_1^{\max \{ r_1, t_1\}} ... \p_m^{\max\{r_m, t_m\}})|}.
\]
\end{proof}

We now proceed to bound the regular terms in the trace formula. By Lemma \ref{lemglobalreg}
\[
I_{reg}(f)=\sum_{\xi \in S(\Omega, \N, 1_{\p})}I(\xi, f)
\]
where $S(\Omega, \N, 1_{\p})$ is the set of $\xi \in \epsilon NE^\times$ such that $v(\xi)\geq 1$ for all $v | \N$ and $(1-\xi)^{-1} \in (\c(\Omega)\d_{E/F})^{-1}$. The set $S(\Omega, \N, 1_{\p})$ is approximately of size $\frac{c(\Omega)d_{E/F}}{|\N|}$. However in the following Lemma we can take advantage of the bounds on the local orbital integrals given in Lemma \ref{lemma:finalbound}, which improve as $c(\Omega)$ increases, to get a good bound on $I_{reg}(f)$.

\begin{lemma}
Assume that $F$, $E$ and $\k$ are fixed and $f_\p=\1_\p$. Then for all $\epsilon>0$,
\[
\left| I_{reg}(f) \right| \ll_{F, E, \k, \eps} \frac{c(\Omega)^\eps}{|\N|}.
\]
\label{lemQ}
\end{lemma}

\begin{proof}
Let $\mathfrak A_1, ... , \mathfrak A_{h_F}$ be a fixed set of representatives of the ideal classes of $F$. Then for any ideal  $\mathfrak I$ of $\O_F$ there exists an $i$ such that $\mathfrak A_i \mathfrak I$ is principal. Let $a \in F$ be such that
\[
a \O_F=\mathfrak A_i\c(\Omega)\d_{E/F}
\]
for some $i$. Then
\begin{align}\label{eqn:norma}
\left|\frac{N_{F/\Q}(a)}{c(\Omega)d_{E/F}}\right |\leq c(F) = \max \{ |\mathfrak A_i| : 1\leq i\leq h_F\}.
\end{align}

For $\xi\in S(\Omega, \N, 1_{\p})$ we have,
\[
\xi = \xi_y := \frac{y}{y-a},
\]
for some $y\in\N$ such that $|\iota_v(y)| < |\iota_v(a)|$ for all $v \in \Sigma_\infty$. Hence we can partition these $y$ into the sets $S_{(r_i),(t_i)}(a)$. Thus
\begin{align}\label{eqn:sum}
I_{reg}(f) =\sum_{r_1\geq 0}\sum_{r_2\geq 0}...\sum_{r_m\geq 0} \sum_{t_1 \geq 0}\sum_{t_2 \geq 0}... \sum_{t_m \geq 0}
\sum_{y \in S_{(r_i),(t_i)}(a)} I(\xi_y, f^{S(\Omega)}) \prod_{i=1}^m I(\xi_y,f_{\p_i}).
\end{align}
First we consider the integrals $I(\xi_y, f^{S(\Omega)})$.
For an ideal $\a \subset \O_F$ we define
\[
R_E^{S(\Omega)}(\a) = \{ \mathfrak b \subset \O_E: N_{E_v/F_v}(\mathfrak b\O_{E_v})= \a\O_{F_v} \text{ for } v \notin S(\Omega) \text{ and } \p_{v}\nmid\mathfrak b \text{ for } v \in  S(\Omega)\}.
\]
By the proof of Proposition
\ref{propglobalunram} we see that
\begin{align} \label{eqn:away1}
|I(\xi_y, f^{S(\Omega)})| \leq C(E) |R_E^{S(\Omega)}(y\N^{-1})||R_E^{S(\Omega)}((y-a))||I(\xi, f_{\Sigma_\infty})|,
\end{align}
where $C(E)$ is a constant depending only on $E$. Furthermore, from Lemma \ref{lemRe3}, it is clear
that
\begin{align} \label{eqn:inft}
|I(\xi_y, f_{\Sigma_\infty})| \leq C(\k),
\end{align}
where $C(\k)$ is a constant
depending only on $\k$.

Let $\sigma_0(x)$ denote the number of divisors of $x$. For any ideal $\a \in\O_F$,
\begin{align} \label{eqn:divisor}
|R_E^{S(\Omega)}(\a)| \leq \sigma_0(N_{F/\Q}(\a))^{[F:\Q]}\ll_\eps |N_{F/\Q}(\a)|^\eps.
\end{align}
For any $y \in S_{(r_i), (t_i)}(a)$, $|\iota_v(y)|< |\iota_v(a)|$ and $|\iota_v(y-a)|< |\iota_v(a)|$ for all $v \in \Sigma_\infty$. Thus
$|N_{F/\Q}(y)|=\prod_{v \in \Sigma_\infty} |\iota_v(y)| \leq |N_{F/\Q}(a)|$ and similarly $ |N_{F/\Q}(y-a)|\leq  |N_{F/\Q}(a)|$.

By  \eqref{eqn:away1}, \eqref{eqn:inft} and \eqref{eqn:divisor},
\begin{align}\label{eqn:away}
|I(\xi_y, f^{S(\Omega)})| \ll_{F, E,\k, \epsilon} |N_{F/\Q}(a)|^\epsilon.
\end{align}
Combining \eqref{eqn:sum} and \eqref{eqn:away} we have,
\begin{align}
|I_{reg}(f)|  \ll_{F,E,\k, \epsilon}  |N_{F/\Q}a|^\epsilon  \sum_{r_1\geq 0}\sum_{r_2\geq 0}...\sum_{r_m\geq 0} \sum_{t_1 \geq 0}\sum_{t_2 \geq 0}... \sum_{t_m \geq 0}
\sum_{y \in S_{(r_i),(t_i)}(a)} \prod_{i=1}^{m} |I(\xi_y, f_{\p_i})|.
\end{align}
To bound $|I(\xi_y, f_{\p_i})|$
we first note that
\begin{align*}
v_{\p_i}(1-\xi)=v_{\p_i}(a)-v_{\p_i}(y-a).
\end{align*}
Thus by Lemma \ref{lemma:finalbound}, for $y \in S_{(r_i),(t_i)}(a)$,
\begin{align} \label{eqn:onebound}
|I(\xi_y, f_{\p_i})| \leq
\begin{cases}
C(E_v, F_v)q_i^{-n_i}L(1, \eta_v)^2 q_i^{\frac{r_i-v_{\p_i}(a)}{2}} & 0 \leq r_i < v_{\p_i}(a)
\\
C(E_v, F_v)q_i^{-n_i}L(1, \eta_v)(1+t_i-r_i) & r_i=v_{\p_i}(a)
\\
C(E_v, F_v)q_i^{-n_i}L(1, \eta_v)(1+r_i-v_{\p_i}(a)) & r_i> v_{\p_i}(a)
\end{cases}
\end{align}
where $C(E_v, F_v)$ is the constant in the statement of Lemma \ref{lemma:finalbound}. In particular, $C(E_v, F_v)=1$ if $v$ is unramified in $E$.

Because the bounds in Lemma \ref{claim:size} and \eqref{eqn:onebound} depend only on $r_i$ and $t_i$ for each $i$,
\begin{align}\label{eqn:reg}
|I_{reg}(f)|  \ll_{F,E,\k, \epsilon}  |N_{F/\Q}(a)|^\epsilon  \frac{|N_{F/\Q}(a)|}{|N_{F/\Q}\N|} \prod_{i=1}^{m} \sum_{r_i\geq 0}\sum_{t_i\geq 0} q_i^{-\max\{r_i, t_i\}} |I(\xi_{y_{r_i,t_i}}, f_{\p_i})|,
\end{align}
where we choose $y_{r_i, t_i}$ to be any element in $F_{\p_i}$ such that $v_{\p_i}(y_{r_i, t_i})=r_i$ and $v_{\p_i}(y_{r_i, t_i}-a)=t_i$.

Again by Lemma \ref{claim:size} and \eqref{eqn:onebound}, for $i$ fixed
\begin{align*}
&\sum_{r_i\geq 0}\sum_{t_i\geq 0} q_i^{-\max\{r_i, t_i\}} |I(\xi_{y_{r_i, t_i}}, f_{\p_i})|
\\
\leq
&\sum_{v_{\p_i}(a)>r_i \geq 0} C(E_v, F_v) q_i^{-n_i}L(1, \eta_v)^2 q_i^{\frac{r_i-v_{\p_i}(a)}{2}}q_i^{-r_i}
\\
+& \sum_{t_i\geq  v_{\p_i}(a)} C(E_v, F_v)q_i^{-n_i} L(1, \eta_v)(1+t_i-v_{\p_i}(a)) q_i^{-t_i}
+\sum_{r_i> v_{\p_i}(a)} C(E_v, F_v) q_i^{-n_i} L(1, \eta_v)(1+r_i-v_{\p_i}(a)) q_i^{-r_i}
\\
\leq
&\sum_{r_i\geq 0} C(E_v, F_v)q_i^{-n_i}L(1, \eta_v)^2 q_i^{\frac{r_i-v_{\p_i}(a)}{2}}q_i^{-r_i}
+2 \sum_{\ell\geq  v_{\p_i}(a)} C(E_v, F_v)q_i^{-n_i} L(1, \eta_v)(1+\ell-v_{\p_i}(a)) q_i^{-\ell}
\\
& \leq C(E_v, F_v)L(1, \eta_v)q_i^{-n_i-\frac{v_{\p_i}(a)}{2}}\left[\frac{L(1, \eta_v)}{1-q_i^{-1/2}}
+ \frac{2q_i^{\frac{-v_{\p_i}(a)}{2}}}{(1-q_i^{-1})^2} \right]
\\
&\leq 2^5 q_i^{-n_i-\frac{v_{\p_i}(a)}{2}}C(E_v, F_v)\\
&\leq 2^5 q_i^{-2 n_i}C(E_v, F_v).
\end{align*}

Finally we can bound the product of these terms over $1\leq i\leq m$. We note that
\[
2^{|S(\Omega)|} \leq \sigma_0(c(\Omega))^{[F: \Q]}\ll_\eps c(\Omega)^\eps.
\]
Also, because $C(E_v, F_v)=1$ unless $v$ ramifies in $E$,
\[
\prod_{i=1}^{m} C(E_v, F_v) \ll_{E, F} 1.
\]
By these facts,
\begin{align}\label{eqn:oneprime}
\prod_{i=1}^m \sum_{r_i}\sum_{t_i} q_i^{-\max\{r_i, t_i\}} |I(\xi_y, f_{\p_i})| \ll_{\epsilon, F, E} \frac{c(\Omega)^\eps} {c(\Omega)}.
\end{align}
Combining \eqref{eqn:reg} and \eqref{eqn:oneprime} and applying \eqref{eqn:norma} we conclude
\[
|I_{reg}(f)| \ll_{F, E, \k, \eps}
\frac{{c(\Omega)}^\eps}{|\N|} .
\]
\end{proof}

Finally we combine the spectral expansion for $I(f)$, Proposition \ref{propspectral}, together with the calculation of $I_{irreg}(f)$, Proposition \ref{propirreg}, with the bound on $I_{reg}(f)$ established above. In the theorem below the term $|\N|^{1+\eps}c(\Omega)^\eps$ comes from the irregular term and $|\N|^{\eps}c(\Omega)^{\frac{1}{2}+\epsilon}$ comes from the bounds on the regular orbital integrals.

\begin{theorem} \label{thmsub}
Fix a totally real number field $F$ and a CM extension $E$ of $F$. Let $\N$ be a squarefree ideal in $\O_F$ such that the number of primes dividing $\N$ has the same parity as $[F: \Q]$ and such that each prime of $F$ dividing $\N$ is inert and unramified in $E$. Let $\Omega$ be a character of $\A_F^\times E^\times \bs \A_E^\times$ which is unramified above $\N$ and has weights at the archimedean places strictly less than $\k$. Then for any $\epsilon>0$,
\[
L_{fin}(1/2, \pi \times \sigma_\Omega) \ll_{F, E, \k, \eps} |\N|^{1+\eps}c(\Omega)^\eps + |\N|^{\eps}c(\Omega)^{\frac{1}{2}+\epsilon},
\]
for all $\pi \in \F(\N, 2\k)$.
\end{theorem}

\begin{proof}
We note that with the weight $\k$ fixed there are only finitely many possibilities for the
archimedean type of $\Omega$, hence it suffices to prove the same bound for the completed
$L$-function.

We take $f_\p=\1_\p$ then from the spectral expansion for $I(f)$ we have, by Proposition
\ref{propspectral},
\begin{align} \label{equation:spec}
I(f) \geq \frac{ L_{S(\Omega)}(1, \eta)^2}{2 |\Delta_F|^2\sqrt{d_{E/F}c(\Omega)}} \frac{4^{[F:\Q]}}{|\N|}  \binom{2\k - 2}{\k + \m - 1}
\sum_{\pi\in \F(\N, 2\k)} \frac{L(1/2, \pi_E\otimes\Omega)}{L(1, \pi, Ad)}.
\end{align}

On the geometric side, we recall that we have written
\begin{align} \label{eqn:irregandreg}
I(f) = I_{irreg}(f) + I_{reg}(f)
\end{align}
in Section
\ref{globalcalculations}. By Proposition \ref{propirreg} we have,
\begin{align}\label{equation:geomirreg}
|I_{irreg}(f)| \ll_{F,E, \eps} \frac{c(\Omega)^\eps}{\sqrt{c(\Omega)}},
\end{align}
and by Lemma \ref{lemQ} we have, for all $\eps > 0$,
\begin{align}\label{equation:geomreg}
|I_{reg}(f)| \ll_{F, E,\k, \eps} \frac{c(\Omega)^\epsilon}{|\N|}.
\end{align}
Hence, combining  \eqref{equation:spec}, \eqref{eqn:irregandreg}, \eqref{equation:geomirreg} and \eqref{equation:geomreg}, and noting that
\[
\frac{1}{L_{S(\Omega)}(1, \eta)} \leq 2^{2|S(\Omega)|} \leq \sigma_0(c(\Omega))^{2[F:\Q]} \ll_\eps c(\Omega)^\eps,
\]
we get,
\begin{align}\label{equation:both}
\sum_{\pi\in\F(\N, 2\k)} \frac{L(1/2, \pi_E\otimes\Omega)}{L(1, \pi, Ad)} \ll_{F, E, \k, \eps} |\N|c(\Omega)^\eps+
c(\Omega)^{\frac{1}{2}+\epsilon}.
\end{align}

Now using positivity of $L(1/2, \pi_E\otimes\Omega)$, which is clear from the period formula, we have
\begin{align}\label{equation:pos}
\frac{L(1/2, \pi_E\otimes\Omega)}{L(1, \pi, Ad)} \leq \sum_{\pi\in\F(\N, 2\k)} \frac{L(1/2,
\pi_E\otimes\Omega)}{L(1, \pi, Ad)}
\end{align}
for any $\pi \in \F(\N, 2\k)$. We now use that for all $\eps>0$,
\begin{align}\label{equation:ad}
L(1, \pi, Ad) \ll_{\k, \eps}  |\N|^{\eps};
\end{align}
cf \cite[Theorem 5.41]{iwaniec:2004}.
Hence by \eqref{equation:both}, \eqref{equation:pos} and \eqref{equation:ad} we have, for all
$\eps>0$,
\[
L(1/2, \pi_E\otimes\Omega) \ll_{\k,\eps}  |\N|^{\eps} \frac{L(1/2, \pi_E\otimes\Omega)}{L(1, \pi,
Ad)}\ll_{F, E, \k, \eps}  |\N|^{1+\eps}c(\Omega)^\eps+|\N|^\eps c(\Omega)^{\frac{1}{2}+\epsilon}.
\]
\end{proof}

Finally we explicate how this Theorem gives a subconvex bound as $\N$ and $\Omega$ vary in certain ranges.

\begin{corollary}
For $0 \leq t < \frac 16$ and $\eps > 0$,
\[
L_{fin}(1/2, \pi \times \sigma_\Omega)\ll_{F, E, \k, \eps} (c(\Omega)|\N|)^{\frac{1}{2}-t},
\]
for $\pi\in\F(\N, 2\k)$; with $\N$ and $\Omega$ satisfying the conditions of Theorem \ref{thmsub} and
\[
c(\Omega)^{\frac{2 t + \eps}{1-(2 t + \eps)}}\leq |\N|\leq c(\Omega)^{\frac{1-(2t+\eps)}{1+2 t + \eps}}.
\]
\label{corsub}
\end{corollary}
\begin{proof}
To beat convexity we need,
\[
|\N|^{1+\eps}c(\Omega)^\eps \leq |\N|^{\frac 12 - t}  c(\Omega)^{\frac 12 - t}
\]
and
\[
|\N|^{\eps}c(\Omega)^{\frac 12 + \eps } \leq |\N|^{\frac 12 - t} c(\Omega)^{\frac 12 - t}.
\]
Thus we need
\[
c(\Omega)^{\frac{2\eps +2t}{1 -2t -2\eps}} \leq |\N| \leq c(\Omega)^{\frac{1-2t-2\eps}{1+2t+2\eps}}.
\]
We note that the range for $|\N|$ is
non-empty provided $\frac 16>t\geq 0$.
\end{proof}

\subsection{Classical reformulation}
\label{classical}

To finish we work out Theorem \ref{thmlargelevel} in the case $F=\Q$ classically. We fix an imaginary quadratic field $E =
\Q(\sqrt{-d})$ of discriminant $-d$ and let $\chi_{-d} = \left( \frac{-d}{\cdot} \right)$ denote the
associated quadratic Dirichlet character.

Let $N$ be a squarefree integer which is the product of an odd number of primes $p$ satisfying
$\chi_{-d}(p) = - 1$. For a positive integer $k$ we denote by $\F(N, 2k)$ the finite set of normalized
newforms of level $N$, weight $2k$, trivial nebentypus and which are eigenforms for all the Hecke
operators. On $\F(N, 2k)$ we take the Petersson inner product defined by
$$
(f, f) = \int_{\Gamma_0(N)\bs\H} |f(x + iy)|^2 \ y^{2k} \ \frac{dx \ dy}{y^2}.
$$
We note that,
$$
L(1, \pi_f, Ad) = \frac{2^{2k}}{N} (f, f)
$$
where $\pi_f$ denotes the automorphic representation of $\GL(2, \A_\Q)$ generated by $f$. We denote
by $L(s, f)$ the completed $L$-function of $f$, which has a functional equation relating the value
at $s$ to $2k - s$.

We now fix a character $\Omega : E^\times \bs \A_E^\times \to \C$ whose restriction to
$\A_\Q^\times$ is trivial. At infinity we have
$$
\Omega_\infty : z \mapsto \left( \frac{z}{\overline{z}} \right)^m,
$$
with $m\in\Z$. We recall that when $\Omega$ does not factor through the norm map $N : \A_E^\times
\to \A_\Q^\times$ there is a modular form $g_\Omega$ of level $d c(\Omega)$, weight $2|m| + 1$ and
nebentypus $\chi_{-d}$ such that
$$
L(s, g_\Omega) = L(s, \Omega).
$$
For $f \in \F(N, 2k)$ we let $L(s, f \times g_\Omega)$ denote the completed Rankin-Selberg
$L$-function which satisfies a functional equation relating the value at $s$ to $2k + 2|m| + 1 - s$.

We recall the well known facts that for the completed $L$-functions, $L(2, 1_\Q) = \frac{\pi}{6}$ and
$$
L(1, \chi_{-d}) = \frac{h_{-d}}{u_{-d} \sqrt{d}},
$$
where $h_{-d}$ denotes the class number of $E$ and $u_{-d} = \# \O_E^\times/\{\pm1\}$.

Taking into account that the gamma factor for  $L(k, f)$ is $2 (2\pi)^{-k}\Gamma(k)$ and the gamma factor for $L(k+|m|+\frac 12, f\times g_\Omega)$ is $4(2\pi)^{-2k}\Gamma(k+m)\Gamma(k-m)$, we now apply Theorem \ref{thmlargelevel} to get the following following averages for the finite parts of the $L$-functions.

\begin{theorem}
Let $N$ be a squarefree integer as above with $N > d$. Then we have
\[
\frac{u_{-d}  \sqrt{d}}{8 \pi^2} \sum_{f\in\F(N, 2)} \frac{L_{fin}(1, f) L_{fin}(1, f \otimes \chi_{-d})}{(f,
f)} = h_{-d} \left( 1-\frac{12h_{-d}}{u_{-d}  \phi(N)} \right),
\]
where $\phi$ denotes the Euler totient function. When $k > 1$ we have
\[
\frac{(2k-2)! u_{-d}\sqrt{d} }{2\pi(4\pi)^{2k-1}} \sum_{f\in\F(N, 2k)} \frac{L_{fin}(k, f) L_{fin}(k,
f\otimes\chi_{-d})}{(f, f)} = h_{-d}.
\]
For a character $\Omega$ as above which does not factor through the norm we get for $N > d c(\Omega)$,
\[
\frac{(2k-2)!u_{-d} \sqrt{d}L_{S(\Omega)}(1, \chi_{-d})}{2\pi (4\pi)^{2k-1}}  \sum_{f\in\F(N, 2k)}
\frac{L_{fin}(k + |m| + \frac 12, f \times g_\Omega)}{(f, f)} = h_{-d}.
\]
\label{thmclassical}
\end{theorem}

We note that when the level is prime the first part of this Theorem agrees with Duke's asymptotic result \cite[Proposition 2]{duke:1995} and the first and second parts agree with Michel and Ramakrishnan's exact formula \cite{ramakrishnan:exact}. One can see \cite[p.5]{ramakrishnan:exact} for explicit examples verifying that this formula agrees with known data.

\providecommand{\bysame}{\leavevmode\hbox to3em{\hrulefill}\thinspace}
\providecommand{\MR}{\relax\ifhmode\unskip\space\fi MR }
\providecommand{\MRhref}[2]{%
  \href{http://www.ams.org/mathscinet-getitem?mr=#1}{#2}
}
\providecommand{\href}[2]{#2}


\begin{thebibliography}{KMV02}

\bibitem[Bla06]{blasius:2006}
Don Blasius, \emph{Hilbert modular forms and the {R}amanujan conjecture},
  Noncommutative geometry and number theory, Aspects Math., E37, Vieweg,
  Wiesbaden, 2006, pp.~35--56. \MR{MR2327298}

\bibitem[Cog03]{cogdell:2003}
James~W. Cogdell, \emph{On sums of three squares}, J. Th\'eor. Nombres Bordeaux
  \textbf{15} (2003), no.~1, 33--44, Les XXII\`emes Journ\'ees Arithmetiques
  (Lille, 2001). \MR{MR2018999 (2005d:11072)}

\bibitem[DFI94]{duke:1994}
W.~Duke, J.~B. Friedlander, and H.~Iwaniec, \emph{Bounds for automorphic
  {$L$}-functions. {II}}, Invent. Math. \textbf{115} (1994), no.~2, 219--239.
  \MR{MR1258904 (95a:11044)}

\bibitem[Duk95]{duke:1995}
William Duke, \emph{The critical order of vanishing of automorphic
  {$L$}-functions with large level}, Invent. Math. \textbf{119} (1995), no.~1,
  165--174. \MR{MR1309975 (95k:11075)}

\bibitem[GP91]{gross:1991}
Benedict~H. Gross and Dipendra Prasad, \emph{Test vectors for linear forms},
  Math. Ann. \textbf{291} (1991), no.~2, 343--355. \MR{MR1129372 (92k:22028)}

\bibitem[Gro87]{gross:1987}
Benedict~H. Gross, \emph{Heights and the special values of {$L$}-series},
  Number theory (Montreal, Que., 1985), CMS Conf. Proc., vol.~7, Amer. Math.
  Soc., Providence, RI, 1987, pp.~115--187. \MR{MR894322 (89c:11082)}

\bibitem[HM06]{harcos:2006}
Gergely Harcos and Philippe Michel, \emph{The subconvexity problem for
  {R}ankin-{S}elberg {$L$}-functions and equidistribution of {H}eegner points.
  {II}}, Invent. Math. \textbf{163} (2006), no.~3, 581--655. \MR{MR2207235}

\bibitem[IK04]{iwaniec:2004}
Henryk Iwaniec and Emmanuel Kowalski, \emph{Analytic number theory}, American
  Mathematical Society Colloquium Publications, vol.~53, American Mathematical
  Society, Providence, RI, 2004. \MR{MR2061214 (2005h:11005)}

\bibitem[ILS00]{iwaniec:2000a}
Henryk Iwaniec, Wenzhi Luo, and Peter Sarnak, \emph{Low lying zeros of families
  of {$L$}-functions}, Inst. Hautes \'Etudes Sci. Publ. Math. (2000), no.~91,
  55--131 (2001). \MR{MR1828743 (2002h:11081)}

\bibitem[IS00a]{iwaniec:2000b}
H.~Iwaniec and P.~Sarnak, \emph{Perspectives on the analytic theory of
  {$L$}-functions}, Geom. Funct. Anal. (2000), no.~Special Volume, Part II,
  705--741, GAFA 2000 (Tel Aviv, 1999). \MR{MR1826269 (2002b:11117)}

\bibitem[IS00b]{iwaniec:2000}
Henryk Iwaniec and Peter Sarnak, \emph{The non-vanishing of central values of
  automorphic {$L$}-functions and {L}andau-{S}iegel zeros}, Israel J. Math.
  \textbf{120} (2000), no.~, part A, 155--177. \MR{MR1815374 (2002b:11115)}

\bibitem[Jac86]{jacquet:1986}
Herv\'e Jacquet, \emph{Sur un r\'esultat de {W}aldspurger}, Ann. Sci. \'Ec.
  Norm. Sup\'er., IV. \textbf{19} (1986), no.~2, 185--229.

\bibitem[JC01]{jacquet:2001}
Herv{\'e} Jacquet and Nan Chen, \emph{Positivity of quadratic base change
  {$L$}-functions}, Bull. Soc. Math. France \textbf{129} (2001), no.~1, 33--90.
  \MR{MR1871978 (2003b:11048)}

\bibitem[JL70]{jacquet:1970}
Herv\'e Jacquet and Robert~P. Langlands, \emph{Automorphic forms on {${\rm
  GL}(2)$}}, Springer-Verlag, Berlin, 1970, Lecture Notes in Mathematics, Vol.
  114. \MR{MR0401654 (53 \#5481)}

\bibitem[KMV02]{kowalski:2002}
E.~Kowalski, P.~Michel, and J.~VanderKam, \emph{Rankin-{S}elberg
  {$L$}-functions in the level aspect}, Duke Math. J. \textbf{114} (2002),
  no.~1, 123--191. \MR{MR1915038 (2004c:11070)}

\bibitem[Lan94]{lang:1994}
Serge Lang, \emph{Algebraic number theory}, second ed., Graduate Texts in
  Mathematics, vol. 110, Springer-Verlag, New York, 1994. \MR{MR1282723
  (95f:11085)}

\bibitem[Mic04]{michel:2004}
Phillipe Michel, \emph{The subconvexity problem for {R}ankin-{S}elberg
  {$L$}-functions and equidistribution of {H}eegner points}, Ann. of Math. (2)
  \textbf{160} (2004), no.~1, 185--236. \MR{MR2119720 (2006d:11048)}

\bibitem[MR]{ramakrishnan:exact}
Philippe Michel and Dinakar Ramakrishnan, \emph{Consequences of the
  {G}ross/{Z}agier formulae: stability of average {$L$}-values, subconvexity
  and non-vanishing mod $p$}, to appear in the memorial volume for Serge Lang,
  Springer-Verlag (2009).

\bibitem[MV07]{michel:2007}
Philippe Michel and Akshay Venkatesh, \emph{Heegner points and non-vanishing of
  {R}ankin/{S}elberg {$L$}-functions}, Analytic number theory, Clay Math.
  Proc., vol.~7, Amer. Math. Soc., Providence, RI, 2007, pp.~169--183.
  \MR{MR2362200}

\bibitem[MW09]{me:periods}
Kimball Martin and David Whitehouse, \emph{Central {$L$}-values and toric
  periods for {${\rm GL}(2)$}}, Int. Math. Res. Not. IMRN (2009), no.~1, Art.
  ID rnn127, 141--191. \MR{MR2471298}

\bibitem[Pop06]{popa:2006}
Alexandru~A. Popa, \emph{Central values of {R}ankin {$L$}-series over real
  quadratic fields}, Compos. Math. \textbf{142} (2006), no.~4, 811--866.
  \MR{MR2249532}

\bibitem[Roy00]{royer:2000}
Emmanuel Royer, \emph{Facteurs {$\bold Q$}-simples de {$J\sb 0(N)$} de grande
  dimension et de grand rang}, Bull. Soc. Math. France \textbf{128} (2000),
  no.~2, 219--248. \MR{MR1772442 (2001j:11041)}

\bibitem[RR05]{ramakrishnan:2005}
Dinakar Ramakrishnan and Jonathan Rogawski, \emph{Average values of modular
  {$L$}-series via the relative trace formula}, Pure Appl. Math. Q. \textbf{1}
  (2005), no.~4, 701--735. \MR{MR2200997}

\bibitem[Sar87]{sarnak:1987}
Peter Sarnak, \emph{Statistical properties of eigenvalues of the {H}ecke
  operators}, Analytic number theory and Diophantine problems (Stillwater, OK,
  1984), Progr. Math., vol.~70, Birkh\"auser Boston, Boston, MA, 1987,
  pp.~321--331. \MR{MR1018385 (90k:11056)}

\bibitem[Sch02]{schmidt:2002}
Ralf Schmidt, \emph{Some remarks on local newforms for {$\rm GL(2)$}}, J.
  Ramanujan Math. Soc. \textbf{17} (2002), no.~2, 115--147. \MR{MR1913897
  (2003g:11056)}

\bibitem[Ser62]{serre:1962}
Jean-Pierre Serre, \emph{Corps locaux}, Publications de l'Institut de
  Math\'ematique de l'Universit\'e de Nancago, VIII, Actualit\'es Sci. Indust.,
  No. 1296. Hermann, Paris, 1962. \MR{MR0150130 (27 \#133)}

\bibitem[Ser97]{serre:1997}
\bysame, \emph{R\'epartition asymptotique des valeurs propres de l'op\'erateur
  de {H}ecke {$T\sb p$}}, J. Amer. Math. Soc. \textbf{10} (1997), no.~1,
  75--102. \MR{MR1396897 (97h:11048)}

\bibitem[Ven05]{venkatesh:2005}
Akshay Venkatesh, \emph{Sparse equidistribution problems, period bounds, and
  subconvexity}, Preprint, 2005.

\bibitem[Wal85]{waldspurger:1985}
Jean-Loup Waldspurger, \emph{Sur les valeurs de certaines fonctions {$L$}
  automorphes en leur centre de sym\'etrie}, Compositio Math. \textbf{54}
  (1985), no.~2, 173--242. \MR{MR783511 (87g:11061b)}

\bibitem[Xue06]{xue:2006}
Hui Xue, \emph{Central values of {R}ankin {$L$}-functions}, Int. Math. Res.
  Not. (2006), Art. ID 26150, 41. \MR{MR2249999}

\bibitem[Zha01]{zhang:2001}
Shou-Wu Zhang, \emph{Gross-{Z}agier formula for {${\rm GL}\sb 2$}}, Asian J.
  Math. \textbf{5} (2001), no.~2, 183--290. \MR{MR1868935 (2003k:11101)}

\end{thebibliography}
\end{document}